\documentclass[11pt]{amsart}
\usepackage{amssymb,amsmath,mathtools}
\usepackage{enumerate}
\usepackage{xcolor}
\usepackage{caption}
\usepackage[centering]{geometry}
\usepackage{url}
\usepackage{hyperref}

\usepackage{graphicx}
\usepackage{color}
\usepackage{amsmath}
\usepackage{amsfonts}
\usepackage{amssymb}
\usepackage{amscd}
\usepackage{framed}
\usepackage{bbm}
\usepackage[most]{tcolorbox}

\usepackage{amsthm}

\newcommand{\R}{\mathbb{R}}
\newcommand{\inr}[1]{\left\langle #1 \right\rangle}

\newcommand{\W}{\mathcal{W}}

\newtheorem{theorem}{Theorem}[section]

\newtheorem{lemma}[theorem]{Lemma}
\newtheorem{corollary}[theorem]{Corollary}
\theoremstyle{definition}
\newtheorem{definition}[theorem]{Definition}
\newtheorem{assumption}[theorem]{Assumption}

\newtheorem{remark}[theorem]{Remark}

\newcommand{\E}{\mathbb{E}}
\renewcommand{\P}{\mathbb{P}}
\newcommand{\C}{\mathrm{\mathbb{C}ov}}

\renewcommand{\tilde}{\widetilde}

\numberwithin{equation}{section}

\begin{document}

\title{On the structure of marginals in high dimensions}

\author{Daniel Bartl}
\address{Department of Mathematics, Department of Statistics and Data Science, National University of Singapore}
\email{bartld@nus.edu.sg}
\author{Shahar Mendelson}
\address{Department of Mathematics,  Texas A\&M University}
\email{shahar.mendelson@gmail.com}
\date{\today}

\begin{abstract}
Let $G, G_1,\dots,G_N$ be independent copies of a standard gaussian random vector in $\mathbb{R}^d$ and denote by $\Gamma = \sum_{i=1}^N  \langle G_i,\cdot\rangle e_i$ the standard gaussian ensemble. We show that, for any set $A\subset S^{d-1}$, with exponentially high probability,
\[
\sup_{x\in A} \frac{1}{N}\sum_{i=1}^N \big| (\Gamma x)^\sharp_i - q_i\big|
\le  c \frac{ \mathbb{E} \sup_{x\in A} \langle G,x\rangle  +  \log^2N }{\sqrt N }.
\]
Here each $q_i$ is  the $\frac{i}{N+1}$--quantile of the standard normal distribution and $(\Gamma x)^\sharp $ denotes the monotone increasing rearrangement of the vector $\Gamma x$.
The estimate is sharp up to a possible logarithmic factor and significantly extends previously known bounds. Moreover, we show that similar estimates hold in much greater generality: after replacing the gaussian quantiles by the appropriate ones, the same phenomenon persists for a broad class of random vectors.
\end{abstract}

\maketitle
\setcounter{equation}{0}
\setcounter{tocdepth}{1}

\section{Introduction}

The problem we explore here revolves around \emph{structure preservation}: how much structure is preserved by i.i.d.\ sampling. 
More accurately, let $\mathcal{F} \subset L_2(\mu)$ be a class of mean-zero functions and set $\sigma=(X_1,...,X_N)$ to be distributed according to $\mu^{\otimes N}$. 
Given the sample $\sigma$, each $f \in \mathcal{F}$ is associated with the random vector $P_\sigma f = \left(f(X_1),...,f(X_N)\right) \in \R^N$, and the question is whether  $P_\sigma f$ captures key features of $f$ --- uniformly in the class $\mathcal{F}$. 

To give a flavour of why this question is of interest, let $A \subset \R^d$, identify $x \in A$ with the linear functional $f_x=\inr{\cdot,x}$ and set $\mathcal{F}_A=\{\inr{\cdot,x} : x \in A\}$. 
Consider a centred random vector $X$ in $\R^d$ and let $\Gamma = \sum_{i=1}^N \inr{X_i,\cdot}e_i$ be the random matrix whose rows are $X_1,...,X_N$ --- independent copies of $X$. 
In this case,  $P_\sigma f_x = \Gamma x$, and the question is whether the elements of $\Gamma A =\{\Gamma x : x \in A\}$ capture the essence of the points in $A$; when that happens, the random image $\Gamma A$ inherits much of $A$'s geometric structure. 

Let us stress that even in this simplified setup, and in the best of scenarios --- when $X$ is the standard gaussian random vector in $\R^d$---, the interplay between $\Gamma A$ and $A$ is not fully understood. Accurate information is known only for notions of structure preservation that are `coarse'. 
For example, the ideal result for what is arguably the most natural notion of structure preservation---the $L_2$ sense---implies that with high probability with respect to $\mu^{\otimes N}$, for every $f \in \mathcal{F}$,
\begin{equation} \label{eq:L-2-iso-intro}
\left|\frac{1}{N}\sum_{i=1}^N f^2(X_i)  - \E f^2 \right| 
\leq c \sup_{f \in \mathcal{F}} \|f\|_{L_2} \cdot \frac{{\rm Comp}(\mathcal{F})}{\sqrt{N}}.
\end{equation}
Here, $c$ is an absolute constant and ${\rm Comp}(\mathcal{F})$ captures the complexity of the class $\mathcal{F}$ in some appropriate sense.
Taking into account the behaviour of the supremum as $N \to \infty$, one would hope that 
${\rm Comp}(F) $ is the expectation of the supremum of the canonical gaussian process indexed by $\mathcal{F}$; but for finite $N$ such a gaussian bound is often false.

The reason that \eqref{eq:L-2-iso-intro} is coarse is because it does not say much about the location of each $\frac{1}{\sqrt N} P_\sigma f$---only that for a typical sample it `lives' in an annulus of radii $\|f\|_{L_2} \pm  \sup_{f \in \mathcal{F}} \|f\|_{L_2}\frac{{\rm Comp}(\mathcal{F})}{\sqrt{N}}$, nothing more.

\begin{remark}
One should keep in mind that `coarse' does not mean `easy' or `pointless'. Establishing \eqref{eq:L-2-iso-intro} (and doing so with the right notion of ${\rm Comp}(\mathcal{F})$) is a notoriously difficult problem with a wide variety of important applications---see, for example, \cite{koltchinskii2017concentration,mendelson2010empirical,mendelson2007reconstruction}.
In fact, satisfactory versions of \eqref{eq:L-2-iso-intro} exist either when the class $\mathcal{F}$ consists of light-tailed random variables---e.g., exhibiting a subgaussian tail decay, or when $\mathcal{F}$ has plenty of symmetries. The most natural example of the latter is $\mathcal{F}_{S^{d-1}} = \{ \inr{\cdot,x} : x \in S^{d-1}\}$, and then \eqref{eq:L-2-iso-intro} captures the extremal singular values of the random matrix $\frac{1}{\sqrt N}\Gamma= \frac{1}{\sqrt{N}}\sum_{i=1}^{N}\inr{X_i,\cdot}e_{i}$;  see \cite{adamczak2010quantitative,bai1993limit,guedon2017interval,mendelson2014singular,tikhomirov2018sample} and the  references therein for some relatively recent results. 
\end{remark}

The notion of `structure preservation' we are interested in here is distributional: that the distribution of the coordinates of each vector $P_\sigma f$ resembles the true distribution of $f$, uniformly in $\mathcal{F}$. And since the joint distribution of the sample $(X_1,\ldots,X_N)$ is invariant under coordinate permutations, it stands to reason that any meaningful notion of `distributional structure preservation' should be formulated up to permutations. Thus, the shape of the vectors $P_\sigma f$ should be studied after rearranging their coordinates, rather than following a specific ordering.

With that in mind, for a vector $v=(v_i)_{i=1}^N\in\R^N$ denote by $v^\sharp$ its monotone nondecreasing rearrangement; in particular, $v^\sharp_1=\min_i v_i$ and $v^\sharp_N=\max_i v_i$. 
With minor abuse of notation, put $f^\sharp(X_i) = v^\sharp_i$ for $v=(f(X_i))_{i=1}^N$.
Setting $(Q_f(u))_{u\in(0,1)}$ to be the quantiles of $f(X)$, the question is whether each reordered vector $(f^\sharp(X_i))_{i=1}^N$ captures $(Q_f(u))_{u\in(0,1)}$, uniformly in $\mathcal{F}$.  

The asymptotic picture for a single random variable is well understood when the distribution of $f(X)$ is sufficiently regular---for example, when $f(X)$ is gaussian. In such cases it is standard to verify (see, e.g. \cite{walker1968note}) that almost surely, as $N\to\infty$,
\[
f^\sharp(X_{\lfloor uN\rfloor}) \to Q_f(u) \qquad \text{for every } u\in(0,1).
\]

\vskip0.3cm
Unfortunately, even if this behaviour passes to a uniform estimate in $\mathcal{F}$ (as is the case under minimal assumptions), this is an asymptotic estimate and provides little information on how well the finite-dimensional geometry of the vectors $\left\{(f^\sharp(X_i))_{i=1}^N :  f \in \mathcal{F}\right\}$ approximate the profiles of the true quantiles. 

As we explain in what follows, existing non-asymptotic results are less than satisfactory. There are known quantitative estimates only for highly structured classes; while in more general scenarios---even when $X$ is a standard gaussian vector in $\R^d$ and $\mathcal{F}_A$ consists of linear functionals indexed by a set $A\subset S^{d-1}$---, what is known is far from sharp.

Our main result is an optimal non-asymptotic bound that holds under minimal assumptions on the class and the underlying measure. Rather than introducing the result in full generality, it is instructive to see what happens in the gaussian case, and by how much the new estimate improves the current state of the art. 

To formulate the gaussian version of our main result, let $F_g(t)=\int_{-\infty}^t \frac{1}{\sqrt{2\pi}} e^{-s^2/2}\,ds$ be the standard gaussian distribution function, put $Q_g=F_{g}^{-1}$, and set $q\in\R^N$ by
\[
q_i = Q_g\!\left(\frac{i}{N+1}\right), \qquad i=1,\ldots,N.
\]
Consider $G$, the standard gaussian vector in $\R^d$, let $X_1,\ldots,X_N$ be independent copies of $G$, and put $\Gamma$ to be the random matrix whose rows are $X_1,\ldots,X_N$.

\begin{tcolorbox}
\begin{theorem}
\label{thm:main.intro}
There exist absolute constants $c_1,c_2>0$ such that the following holds. For every set $A\subset S^{d-1}$, if
\begin{align}
\label{eq:rest.Delta}
\Delta \ge c_1 \frac{\max\left\{ (  \E \sup_{x\in A} \inr{G,x} )^2, \log^4 (N) \right\}}{N},
\end{align}
then with probability at least $1-\exp(-c_2\Delta N)$,
\begin{align}
\label{eq:structure.Gamma.via.L1}
\sup_{x\in A} \frac{1}{N}\sum_{i=1}^N \big|(\Gamma x)^\sharp_i - q_i\big|
\le \sqrt{\Delta}.
\end{align}
\end{theorem}
\end{tcolorbox}

In other words, with high probability, for every $x\in A$ there is a rearrangement of the coordinates of the vector $\Gamma x$ that `lives' in a small $\ell_1(\R^N)$-ball centered in $q$. 
As it happens, the only (potential) looseness is the $\log^4 (N)$ factor; the restriction on $\Delta$ and the probability estimate with which \eqref{eq:structure.Gamma.via.L1} holds are optimal. Indeed, we show that if $N\geq c$, then for every $\Delta>0$ and every fixed $x\in A$,
\[
\P\!\left( \frac{1}{N}\sum_{i=1}^N \big|(\Gamma x)^\sharp_i - q_i\big| \ge \sqrt{\Delta} \right)
\ge \exp(-c'\Delta N).
\]
And also that with probability at least $0.99$,
\[
\sup_{x\in A} \frac{1}{N}\sum_{i=1}^N \big|(\Gamma x)^\sharp_i - q_i\big|
\ge c'' \frac{\E \sup_{x\in A} \inr{G,x}}{\sqrt{N}}.
\]
The proof of the optimality of Theorem \ref{thm:main.intro} can be found in Section \ref{sec:proof.structure}. 

\vskip0.3cm

In addition to being optimal, Theorem \ref{thm:main.intro} is a substantial improvement on existing results: even in the gaussian setting the restrictions on $\Delta$ known previously were considerably weaker. 
For example (and ignoring logarithmic factors), there are sets $A\subset S^{d-1}$ for which, thanks to Theorem \ref{thm:main.intro}, $\Delta \sim 1/N $ is a `legal choice', but previously the best that one could hope for was $\Delta \gtrsim\sqrt d /N $, see \cite{bartl2025empirical}.
Moreover, all prior results (which are stated in or can be derived from \cite{bartl2022structure,bartl2025uniform,boedihardjo2025sharp,gine2006concentration,lugosi2024multivariate,olea2022generalization}) were highly restricted---either requiring strong regularity assumptions on the set $A$, or requiring that $X$ is light-tailed and the marginals have bounded densities. In contrast, the general version of Theorem \ref{thm:main.intro}---formulated in the next section---see Theorem \ref{thm:main.structure}---, is (almost) universal.

At the heart of the proof of Theorem \ref{thm:main.intro} is a uniform concentration estimate on the Wasserstein distance between empirical distribution functions and their true counterparts (see Theorem \ref{thm:main}), a fact that is of independent interest and has several applications (see one example in Appendix \ref{sec:app.application}).

\section{A uniform estimate on the Wasserstein distance}
\label{sec:main.general}
Let $(\Omega,\mu)$ be a probability space, set $X$ to be distributed according to $\mu$, and let $X_1,\dots,X_N$ be independent copies of $X$.
Consider a class of functions $\mathcal{F} \subset L_2(\mu)$, for $f\in\mathcal{F}$ set $F_f(t)=\P(f(X)\leq t)$ to be its distribution function, and let
\[ F_{N,f}(t) = \frac{1}{N}|\{ i \leq N : f(X_i) \leq t\}|, \quad t\in\R\]
be the empirical distribution function of $f(X)$.
For two distribution functions $F$ and $H$, denote their first order Wasserstein distance  by
\[ \W_{1}(F,H)
	=\int_\R |F(t)-H(t)|\,dt.\]
We refer  to \cite{figalli2021invitation,villani2021topics}  for more information on the $\mathcal{W}_{1}$ distance and its role in optimal transport.
Our main result follows from a uniform estimate on the $\W_1$ distance between $F_{N,f}$ and $F_f$.

Set $(G_f)_{f\in\mathcal{F}}$ to be the canonical gaussian process associated with $\mathcal{F}$ (that is, the process is gaussian with mean zero and covariance $\C(G_f,G_h) = \C(f,h)$), and let
	\[ d_{\mathcal{F}}  = \sup_{f\in\mathcal{F}} \| f\|_{L_{2}}
	\quad\text{and}\quad
	 d^{\ast} (\mathcal{F}) = \left( \frac{  \E \sup_{f\in\mathcal{F}} G_f }{ d_{\mathcal{F}} } \right)^{2} \]
be the radius of $\mathcal{F}\cup\{0\}$ and its  critical (or Dvoretzky-Milman) dimension, respectively. 
For a function $f$, set 
\[ \|f\|_{L_{2}(\P_{N})}
= \left(\frac{1}{N}\sum_{i=1}^{N} f^{2}(X_{i}) \right)^{1/2},\]
and the key assumption we require is  that $ \|\cdot \|_{L_{2}}$ and $ \|\cdot \|_{L_{2}(\P_{N})}$ are compatible in the following sense:

\begin{assumption}
\label{ass:main}
	There are  $\theta > \E \sup_{f\in\mathcal{F}} G_f $, $B\geq 1$ and an event $\Omega_{0}$ on which, for every $f,h\in\mathcal{F}\cup\{0\}$, 
	\[ \|f-h\|_{L_{2}(\P_{N})} \leq B   \|f-h\|_{L_{2}}  + \frac{\theta}{\sqrt N}.\]	
\end{assumption}

We show in Section \ref{ref:intro.assumption}  that Assumption \ref{ass:main} is always satisfied with $\theta\sim\E\sup_{f\in\mathcal{F}} G_f$ when the class $\mathcal{F}$ is subgaussian, and in Appendix \ref{sec:app.proof.ass}  that it is valid for other, more heavy-tailed classes as well.

The crucial estimate is as follows:
\begin{tcolorbox}
\begin{theorem}
\label{thm:main}
	For any $B\geq 1$  there are constants $c_{1}$ and $c_{2}$ depending only on $B$ for which the following holds.
	If  Assumption \ref{ass:main} is satisfied and 
	\[   \Delta  N  \geq c_{1}  \max\left\{ \frac{ \theta^{2} }{d_{\mathcal{F}}^{2 }  } ,  \log^{4}(N ) \right\}, \]
	then with probability at least $1-\exp(- c_{2}\Delta N) - \P(\Omega_{0}^c)$,
	\begin{align}
	\label{eq:main.thm.outcome}
	\sup_{f\in\mathcal{F}} \W_1(F_{N,f}, F_f)
	 \leq   d_{\mathcal{F}} \sqrt \Delta.
	 \end{align}
\end{theorem}
\end{tcolorbox}

\begin{remark}
The way $c_1$ and $c_2$ in Theorem \ref{thm:main} depend on $B$ is of secondary importance. The proof we present implies that one may set $c_1=c_1'B^4$ and $c_2=c_2'/B^4$ for absolute constants $c_1',c_2'$, but that is likely to be suboptimal.  
\end{remark}

\begin{remark}
We may assume that each $f\in\mathcal{F}$ has mean zero: the $\W_1$ distance is invariant to centering.
We also ignore all questions related to measurability that can be resolved using standard methods (e.g., assuming that the class can be well-approximated by a countable subset).
\end{remark}

As it turns out, the estimate in Theorem \ref{thm:main} is sharp---even in a gaussian setting and for classes of linear functionals $\mathcal{F}=\mathcal{F}_A=\{ \inr{\cdot,x} : x\in A\}$.

\begin{lemma}
\label{lem:optimality}
	There are absolute constants $c_{1}$ and $c_{2}$ for which the following holds.
	Let $X=G$ be the standard gaussian vector in $\R^d$, and set $A\subset S^{d-1}$.
	Then for every $\Delta \geq \frac{1}{N}$ and $x\in A$, with probability at least $\exp(-c_{1}\Delta N)$,
\begin{align*}
\W_1(F_{N,x}, F_{x})
\geq  c_{2}\sqrt\Delta.
\end{align*}
Moreover,  with probability at least $0.99$,
\begin{align*}
\sup_{x\in A} \W_1(F_{N,x}, F_{x})
	\geq c_{2}\frac{\E \sup_{x\in A} \inr{G,x}}{\sqrt N}.
\end{align*}
\end{lemma}

The proofs of Theorem \ref{thm:main} and Lemma \ref{lem:optimality} are presented in Section \ref{sec:proof.main}.

Theorem \ref{thm:main} leads to our main result (and to its gaussian version in Theorem \ref{thm:main.intro}) thanks to  the standard representation of the $\W_{1}$-distance using quantile functions: for any two distribution functions $F$ and $H$,
	\begin{align}
	\label{eq:W1.rep}
	\begin{split}
	\W_{1}(F,H)
	&= \int_{0}^{1} \left| F^{-1}(u) - H^{-1}(u)\right|\,du. \\ 
	\end{split}
	\end{align}
Moreover,  observe that for every $i=1,\dots, N$ and $u\in [\frac{i-1}{N}, \frac{i}{N})$, $F_{N,f}^{-1}(u) = f^{\sharp}(X_{i}) $.
Thus, setting $q_i(f) = F_f^{-1}(\frac{i}{N+1})$ for $i=1,\dots,N$, it follows from \eqref{eq:W1.rep} that $\mathcal{W}_1(F_{N,f},F_f) $ is almost equal to  $\frac{1}{N} \sum_{i=1}^N | f^\sharp(X_i) -  q_i(f)|$ up to the discrepancy between $F_f^{-1}$ and its discretization---which can be easily controlled (see Lemma \ref{lem:quantile.shift}).
We thus have the following:

\begin{tcolorbox}
\begin{theorem}
\label{thm:main.structure}
For any $B\geq 1$  there are constants $c_{1}$ and $c_{2}$ depending only on $B$ such that the following holds.
	If  Assumption \ref{ass:main} is satisfied and 
	\[   \Delta  N  \geq c_{1}  \max\left\{ \frac{ \theta^{2} }{d_{\mathcal{F}}^{2 }  } ,  \log^{4}(N ) \right\}, \]
	then with probability at least $1-\exp(- c_{2}\Delta N) - \P(\Omega_{0}^c)$,
\begin{align*}
\sup_{f\in\mathcal{F}} \frac{1}{N}\sum_{i=1}^N \big| f^\sharp(X_i)  - q_i(f)\big|
\le d_\mathcal{F} \sqrt{\Delta}.
\end{align*}
\end{theorem}
\end{tcolorbox}

The proofs of Theorem \ref{thm:main.structure} and Theorem \ref{thm:main.intro} are presented in Section \ref{sec:proof.structure}.

\subsection{On Assumption \ref{ass:main}}
\label{ref:intro.assumption}

Consider first the case when the class $\mathcal{F}$ is $L$-subgaussian; that is, for every $f,h\in\mathcal{F}\cup\{0\}$,
\begin{align}
\label{eq:def.subgaussian}
 \|f-h\|_{L_p}\leq L \sqrt p \|f-h\|_{L_2} 
 \quad\text{for } p \geq 2.
\end{align}
Note that an equivalent formulation is that 
\[
\P\left( |f(X)-h(X)| \geq L' \lambda \|f-h\|_{L_2} \right)
\leq 2\exp(-\lambda^2)
\quad \text{for }\lambda\geq 0,
\]
and then $L'\sim L$; we refer e.g.\ to \cite{vershynin2018high} for this and other basic facts about subgaussian distributions.

We then have the following result.

\begin{lemma}
\label{prop:assumption.gaussian}
There are absolute constants $c_{1}$ and $c_2$ for which the following holds.
If the class $\mathcal{F}$ is $L$-subgaussian, $0\in\mathcal{F}$, and $\theta\geq c_1 L \E \sup_{f\in\mathcal{F}} G_f $, then Assumption \ref{ass:main} is satisfied with that $\theta$, $B=c_{1}L$, and an event $\Omega_{0}$ of probability at least $1-\exp(-c_2\max\{ N,\frac{\theta^2}{d^2_{\mathcal{F}}}\})$.
\end{lemma}

The (standard) proof  is given in Appendix \ref{sec:app.proof.ass}. 
Since $d^\ast(\mathcal{F}\cup\{0\})\sim \max\{ 1, d^\ast(\mathcal{F})\}$, Theorem \ref{thm:main} and Lemma \ref{prop:assumption.gaussian} (applied with $\theta\sim d_{\mathcal{F}}\sqrt{\Delta N}$) yield the following.

\begin{corollary}
\label{cor:W1.subgauss}
There are constants $c_{1}$ and $c_2$ depending only on $L$ such that the following holds.
If the class $\mathcal{F}$ is $L$-subgaussian and 
\[
\Delta N \geq c_1 \max\{ d^\ast(\mathcal{F}), \log^4(N)\},
\]
then with probability at least $1-\exp(-c_2\Delta N)$,
\[
\sup_{f\in\mathcal{F}} \W_1(F_{N,f},F_f) \leq d_\mathcal{F} \sqrt\Delta .
\]
\end{corollary}

We also show in Appendix \ref{sec:app.proof.ass} that Assumption \ref{ass:main} holds for various other (possibly heavy-tailed) classes of functions.

\section{Proofs of Theorem \ref{thm:main} and Lemma \ref{lem:optimality}}
\label{sec:proof.main} 

The proof of Theorem~\ref{thm:main} relies on a (non-standard) variant of Talagrand’s generic chaining theory.
For a detailed and illuminating exposition on generic chaining we refer to Talagrand's treasured book 
(see \cite{talagrand2022upper}).

\begin{definition}
	An admissible sequence $(\mathcal{F}_{s})_{s\geq 0}$ of the class $\mathcal{F} \subset L_2(\mu)$ is a sequence of subsets $\mathcal{F}_s \subset \mathcal{F}$ that satisfy $|\mathcal{F}_{0}|=1$ and $|\mathcal{F}_{s}|\leq 2^{2^{s}}$. 

	Talagrand's $\gamma_{2}$-functional (with respect to the $L_2(\mu)$ distance) is
	\[ \gamma_{2}(\mathcal{F}) = \inf_{(\mathcal{F}_{s})_{s\geq 0}} \sup_{f\in\mathcal{F}} \sum_{s\geq 0} 2^{s/2} \| f-\pi_{s}f\|_{L_{2}},\]
	where $\pi_{s}f$ is the nearest element to $f$ in $\mathcal{F}_{s}$ with respect to the $L_{2}(\mu)$ distance and the infimum is taken with respect to all admissible sequences of $\mathcal{F}$.
\end{definition}

Talagrand's majorizing measures theorem \cite{talagrand1987regularity} implies that $\gamma_2(\mathcal{F})$ --- seemingly, a purely metric object---, is actually a probabilistic entity: it is equivalent to the expectation of the supremum of the canonical gaussian process indexed by $\mathcal{F}$. Formally, there are absolute constants $c_{1}$ and $c_{2}$ for which, for every $\mathcal{F} \subset L_2(\mu)$,  
$$
c_{1} \E \sup_{f\in\mathcal{F}} G_{f} \leq \gamma_{2}(\mathcal{F}) \leq c_{2} \E \sup_{f\in\mathcal{F}} G_{f},
$$ 
where $(G_{f})_{f\in\mathcal{F}}$ is the canonical gaussian process indexed by the class $\mathcal{F}$.

The second ingredient we require is the following dual representation of the $\W_1$-distance:
If $F$ and $H$ have finite first moments,
\begin{align}
\label{eq:W1.rep.Lip}
\W_{1}(F,H)
= \sup\left\{ 
\int_{\mathbb{R}} \varphi \, dF - \int_{\mathbb{R}} \varphi \, dH 
:\ \varphi\colon \R \to \R \text{ is $1$-Lipschitz}
\right\}.
\end{align}
In particular, \eqref{eq:W1.rep.Lip} immediately yields the following observation.

\begin{lemma}
\label{lem:W.f.plus.h}
	For every two measurable functions $f$ and $h$,
	\[ \W_{1}(F_{N,f+h}, F_{f+h} )
	\leq  \W_{1}(F_{N,f}, F_{f} ) +  \|h\|_{L_{1}(\P_{N})} + \|h\|_{L_{1}} . \] 
\end{lemma}
\begin{proof}
We assume without loss of generality that $f$ and $h$ have finite first moments.
	Clearly, by the triangle inequality,
	\begin{align*} 
	\W_{1}(F_{N,f+h}, F_{f+h} ) 
	&\leq  \W_{1}(F_{N,f+h}, F_{N,f} ) +  \W_{1}(F_{N,f}, F_{f} ) +  \W_{1}(F_{f}, F_{f+h} ).
	\end{align*}
	Next, let $\varphi$ be any 1-Lipschitz function.
	Then 
	\[ \left| \E \varphi(f(X)) - \E  \varphi(f(X)+h(X)) \right|
	\leq \E[|h(X)|]\]
	and by   \eqref{eq:W1.rep.Lip},  $ \W_{1}(F_{f}, F_{f+h} ) \leq \|h\|_{L_1}$.
	In a similar fashion, 
	\[\W_{1}(F_{N,f+h}, F_{N,f} ) \leq \|h\|_{L_{1}(\P_{N})},\] 
	 which completes the proof.
\end{proof}

\begin{remark}
We may assume that $\Delta \leq 1$.
Indeed, by Lemma \ref{lem:W.f.plus.h} and Jensen's inequality, for every  $f\in\mathcal{F}$,
\[ \W_{1}(F_{N,f}, F_{f} )
	\leq  \|f\|_{L_{2}(\P_{N})} + \|f\|_{L_{2}}, \]
and clearly $\|f\|_{L_2}\leq d_\mathcal{F}$.
Moreover, by Assumption \ref{ass:main}, for every realization $(X_i)_{i=1}^N\in\Omega_0$,
\[\|f\|_{L_{2}(\P_{N})} \leq B d_\mathcal{F} + \frac{\theta}{\sqrt N}\]
and thus  $\W_{1}(F_{N,f}, F_{f} )\leq c(B) d_\mathcal{F}\sqrt\Delta$ for every   $\Delta\geq 1$ satisfying the restriction from Theorem \ref{thm:main}, that $\Delta N \geq c' \theta^2/d_\mathcal{F}^2 $.
\end{remark}

\subsection{Setting up the chaining argument}

If Assumption \ref{ass:main} is satisfied for $\mathcal{F}$, it is also satisfied for $\mathcal{F}\cup\{0\}$, hence we  may assume that $0\in\mathcal{F}$.

Observe that there is a set $\mathcal{F}^{\theta}\subset \mathcal{F}$ with cardinality at most $2^{c N}$ that is  $\theta/\sqrt{N}$-separated (that is, for any  distinct $f,h\in\mathcal{F}^{\theta}$, $\|f-h\|_{L_{2}}\geq \theta/\sqrt N$) and covers $\mathcal{F}$ at scale $2\theta/\sqrt N$ (that is, for every $f\in\mathcal{F}$ there is $h\in\mathcal{F}^\theta$ satisfying $\|f-h\|_{L_2}\leq 2\theta/\sqrt N$).
Indeed, by Sudakov's inequality (see, e.g., \cite[Theorem 5.6]{pisier1999volume}),  if $\mathcal{N}(\varepsilon)$ denotes the smallest cardinality of a set that covers $\mathcal{F}$ at scale $\varepsilon>0$,  then for $\varepsilon = \theta /  \sqrt N $,
\[  \log \left( \mathcal{N} \left( \varepsilon\right)\right)
\leq c_0 \left( \frac{ \E  \sup_{f\in\mathcal{F}} G_f }{\varepsilon} \right)^2 
\leq c_0 N, \]
where the second inequality follows from the condition on $\theta$ in Assumption \ref{ass:main}.
The claim now follows from the standard relation  between packing numbers and covering numbers.
Finally, we may  assume without loss of generality that $0\in\mathcal{F}^\theta$.

Next, let $(\mathcal{F}^{\theta}_{s})_{s\geq 0}$ be an almost optimal admissible sequence of $\mathcal{F}^{\theta}$ and since $\mathcal{F}^{\theta}\subset\mathcal{F}$, we have  that $\gamma_{2}(\mathcal{F}^{\theta})\leq \gamma_{2}(\mathcal{F})$.
Denote by  $s_{1}\geq 0$ the smallest integer that satisfies $2^{2^{s_{1}}} \geq |\mathcal{F}^{\theta}|$; hence $2^{s_{1} }\leq c' N$ and we set $\mathcal{F}^{\theta}_{s}= \mathcal{F}^{\theta}$ for $s\geq s_1$.
For $f\in\mathcal{F}$ and $s\geq 0$,  let $\pi_{s}f $ be the closest element in $\mathcal{F}^{\theta}_{s}$ to $f$. 
In particular, $\pi_s f = \pi_{s_1}f$ for $s\geq s_1$.

Combining these observations we have the following:
\begin{tcolorbox}
For every $f\in\mathcal{F}$, 
\begin{align}
\label{eq:end.of.chain}
\|f-\pi_{s_{1}}f\|_{L_{2}}\leq  \frac{2\theta}{\sqrt N},
\end{align}
 for $ s\leq s_{1}$, 
\begin{align}
\label{eq:middle.of.chain}
\text{either  } \|\Delta_{s }f \|_{L_{2}}=0 \, \text{ or }   \,\|\Delta_{s }f \|_{L_{2}}\geq  \frac{\theta}{\sqrt N},
\end{align}
and 
\begin{align}
\label{eq:almost.optioma.admissible.seq}
 \sup_{f\in\mathcal{F}} \sum_{s\geq 0} 2^{s/2} \| \Delta_{s} f \|_{L_{2}}
\leq 4 \gamma_{2}(\mathcal{F}^{\theta})
\leq 4 \gamma_{2}(\mathcal{F}).
\end{align}
\end{tcolorbox}
Finally, recall that $\Delta\geq \frac{1}{N}$ and set  $s_{0}$ to be the first integer that satisfies $2^{s_{0}}\geq \Delta N$.

\begin{remark}
In what follows we assume that $s_{1}>s_{0}$.
If $s_{1}\leq s_{0}$, the arguments presented below simplify considerably, as there is no need for  chaining; we therefore omit the standard details.
\end{remark}

The first step in the proof of Theorem \ref{thm:main} is the following  reduction.

\begin{lemma}
\label{lem:reduction.to.s1}
Using the notation of Assumption \ref{ass:main}, for every $(X_{i})_{i=1}^{N}\in\Omega_{0}$ and every $f\in\mathcal{F}$, 
\[ \W_{1}(F_{N,f}, F_{f} )
\leq  \W_{1}(F_{N,\pi_{s_{1}} f}, F_{\pi_{s_{1}} f} ) + (2+ B) \frac{\theta}{\sqrt N}.\]
\end{lemma}
\begin{proof}
By Lemma \ref{lem:W.f.plus.h}, for every $f\in\mathcal{F}$
\[ \W_{1}(F_{N,f}, F_{f} )
\leq  \W_{1}(F_{N,\pi_{s_{1}} f}, F_{\pi_{s_{1}} f} ) + \| f-\pi_{s_{1}} f \|_{L_{1}} + \|f-\pi_{s_{1}} f \|_{L_{1}(\P_{N})}.\]
Recall that 
\[  \| f-\pi_{s_{1}} f \|_{L_{1}} 
\leq  \| f-\pi_{s_{1}} f \|_{L_{2}}
 \leq \frac{\theta}{\sqrt N},\]
 and that  by Assumption \ref{ass:main},
\[ \|f-\pi_{s_{1}} f \|_{L_{2}(\P_{N})}
\leq B \| f-\pi_{s_{1}} f \|_{L_{2}} + \frac{\theta}{\sqrt N}, \]
as required. 
\end{proof}

The next step in the proof is truncation.
For each $m>0$ define
\[
\phi( \,\cdot  \, ;\,  m)\colon\R\to\R,
\quad
\phi( x  \, ;\,  m)=
\begin{cases}
x  &\text{if } |x|\leq m,\\
m &\text{if }  x>m, \\
-m  & \text{otherwise},
\end{cases}
\]
and set $\psi( \,\cdot  \, ;\,  m)={\rm id}-\phi( \,\cdot  \, ;\,  m)$.
Note that $|\psi(x\, ;\,  m)| = \max\{0,|x|-m\}$.

With a  minor abuse notation, set for any function $h$,
\begin{align*}
  \phi_{s}(h) 
&=\phi\left(h \, ;\,  \| h   \|_{L_{2}} \sqrt{ N/2^{s} } \right),\\
 \psi_{s}(h) 
&=\psi\left(h \, ;\,  \| h   \|_{L_{2}} \sqrt{ N/2^{s} } \right),
 \end{align*} 
 and  for  $f\in\mathcal{F}$ let 
\begin{align*}
 \mathcal{T}(f) &= \sum_{s=s_{0}}^{s_{1}-1} \phi_{s }(\Delta_{s}f) + \phi_{s_{0}} (\pi_{s_0}f).
 \end{align*}

The idea behind the introduction of $\mathcal{T}(f)$ is  a multi-level truncation.
Clearly $ \pi_{s_{1}} f =  \sum_{s=s_{0}}^{s_{1}-1} \Delta_{s}f +\pi_{s_0}f$; thus $\mathcal{T}(f)$ is a truncation of $\pi_{s_{1}}f$ that takes place for each `link' separately, according to the level $s$ and to the $L_{2}$-norm of the link. 
This is essential in guaranteeing that the wanted degree of concentration is exhibited by all the links simultaneously---and is the reason behind the choice of the truncation level of an $s$-link at $\|\Delta_{s} f\|\sqrt{N/2^{s}}$.
Note that for $s>s_1$, $\sqrt{N/2^s}<1$, meaning that the truncation would be beyond a meaningful scale; that is the reason why the construction is terminated at the level $s_1$.

\begin{lemma}
\label{lem:reduction.to.truncation}
There is an absolute constant $c$ such that for $(X_{i})_{i=1}^{N}\in\Omega_{0}$ and  every $f\in\mathcal{F}$,
\begin{align*}
 \W_{1}(F_{N,\pi_{s_{1}}f}, F_{\pi_{s_{1}}f}) 
\leq \W_{1}(F_{N,\mathcal{T}(f)}, F_{\mathcal{T}(f)} ) + cB^{2}\left(  d_{\mathcal{F}} \sqrt\Delta + \frac{\gamma_{2}(\mathcal{F})}{\sqrt N}\right).
\end{align*} 
\end{lemma}
\begin{proof}
By Lemma \ref{lem:W.f.plus.h},
\begin{align*}
 \W_{1}(F_{N,\pi_{s_{1}}f}, F_{\pi_{s_{1}}f}) 
&\leq \W_{1}(F_{N,\mathcal{T}(f)}, F_{\mathcal{T}(f)} ) \\
&\qquad + \left\| \pi_{s_1}f- \mathcal{T}(f) \right\|_{L_{1}(\P_N )} + \left\|\pi_{s_1} f-\mathcal{T}(f)\right\|_{L_{1}}.
\end{align*}

To estimate $\|\pi_{s_1}f-\mathcal{T}(f) \|_{L_{1}}$, set $b_{s}= \|\Delta_{s} f\|_{L_{2}}\sqrt{N/2^s}$ and by the  Cauchy-Schwarz inequality followed by  Markov's inequality,
\begin{align*}
 \left\| \psi_{s}(\Delta_{s}f) \right\|_{L_{1}}
&\leq\E\left[ |\Delta_{s} f | 1_{\{ |\Delta_{s} f |\geq b_{s} \} } \right] \\
&\leq \|\Delta_{s} f \|_{L_{2}} \sqrt{ \P( |\Delta_{s} f |\geq b_{s}) }  \\
&\leq \|\Delta_{s} f \|_{L_{2}} \sqrt{ \frac{2^{s}}{N} }.
\end{align*}
In a similar manner, using that $2^{s_{0}}\geq  \Delta N$ and that $\|\pi_{s_{0}} f\|_{L_{2}}\leq d_{\mathcal{F}} $,
\begin{align*}
 \left\| \psi_{s_{0}}(\pi_{s_{0}}f) \right\|_{L_{1}}
&\leq \| \pi_{s_{0}} f \|_{L_{2}} \sqrt{ \P( |\pi_{s_{0}} f |\geq \|\pi_{s_{0}} f\|_{L_{2}}/\sqrt{\Delta}) }  
\leq   d_{\mathcal{F}}  \sqrt\Delta.
\end{align*}
Thus, by \eqref{eq:almost.optioma.admissible.seq},
\begin{align}
\label{eq:pi.s1.vs.T.some.eq}
\begin{split}
\| \pi_{s_1} f-\mathcal{T}(f) \|_{L_{1}}
&\leq \left\| \psi_{s_{0}}(\pi_{s_{0}}f) \right\|_{L_{1}} + \sum_{s=s_{0}}^{s_{1}-1} \left\| \psi_{s}(\Delta_{s}f) \right\|_{L_{1}}\\
&\leq  d_{\mathcal{F}} \sqrt\Delta  + \sum_{s=s_{0}}^{s_{1}-1}  \sqrt\frac{2^{s}}{N} \left\|\Delta_{s}f \right\|_{L_{2}}
\leq d_{\mathcal{F}}  \sqrt\Delta  + 4 \frac{\gamma_{2}(\mathcal{F})}{\sqrt N}.
\end{split}
\end{align}

The analysis of the  $\|\psi_{s}(\Delta_{s} f)\|_{L_{1}(\P_N)}$-term is slightly more involved.
Since $\mathcal{F}_{s_{1}}^{\theta}$ is a $\theta/\sqrt{N}$-separated set, we either have $\Delta_{s} f=0$ (in which case  $\| \psi_{s}(\Delta_sf)\|_{L_{1}(\P_{N})} =0$), or $\|\Delta_{s } f\|_{L_{2}}\geq \theta/\sqrt N$.
In the latter case, and since $|\psi_{s}(\Delta_sf)| = \max\{0,|\Delta_sf|-b_{s}\}$,  tail-integration followed by Markov's inequality show that 
\begin{align*}
 \left\| \psi_{s}(\Delta_sf) \right\|_{L_{1}(\P_{N})}
& =\int_{b_{s}}^{\infty}\P_{N}(|\Delta_{s} f|> t) \,dt 
\leq \int_{b_{s}}^{\infty} \frac{ \|\Delta_{s} f \|_{L_2(\P_N)}^{2}}{ t^{2}} \,dt .
\end{align*}
Moreover, $\| \Delta_s f\|_{L_{2}}\geq\theta/\sqrt N$ and it thus follows from Assumption \ref{ass:main} that 
\[ \|\Delta_{s} f \|_{L_2(\P_N)}^{2} 
\leq  (B+1)^{2}  \|\Delta_{s} f \|_{L_2}^{2} .\]
Therefore, by the choice of $b_{s}$,
\begin{align*}
\left\| \psi_{s}(\Delta_sf) \right\|_{L_{1}(\P_{N})}
&\leq (B+1)^{2}  \|\Delta_{s} f \|_{L_2}^{2}  \int_{b_{s}}^{\infty} \frac{1}{t^{2}} \,dt \\
&=(B+1)^{2} \|\Delta_{s} f \|_{L_2} \sqrt{\frac{2^{s}}{N}}.
\end{align*}

The same arguments can be used to show that
\begin{align*}
 \left\| \psi_{s_{0}}(\pi_{s_{0}}f) \right\|_{L_{1}(\P_{N})}
&\leq c_{1}B^{2}  d_{\mathcal{F}}  \sqrt\Delta,
\end{align*}
and following the same path as in \eqref{eq:pi.s1.vs.T.some.eq}, 
\[\| \pi_{s_1}f-\mathcal{T}(f) \|_{L_{1}(\P_{N})} 
\leq c_{2} B^{2} \left(  d_{\mathcal{F}}  \sqrt\Delta  +  \frac{\gamma_{2}(\mathcal{F})}{\sqrt N} \right),\]
as required. 
\end{proof}

\subsection{The heart of the proof}

We are left to deal with $\W_{1}(F_{N,\mathcal{T}(f)}, F_{\mathcal{T}(f)})$, which is where the non-standard chaining argument is required. 

For $f\in\mathcal{F}$ set $S_{s_0-1}(f) = \phi_{s_{0}}(\pi_{s_{0}}f)$ and for $r\geq s_{0}$ let
\[ S_{r}(f) =  \sum_{s=s_{0}}^{r} \phi_{s}(\Delta_{s}f) +  \phi_{s_{0}}(\pi_{s_{0}}f).\]
In particular $\mathcal{T}(f)=S_{s_{1}-1}(f)$ is a telescopic sum:
\[\mathcal{T}(f) =      \sum_{r=s_{0}}^{s_{1}-1} \left( S_{r}(f)-S_{r-1}(f)  \right) +  \phi_{s_{0}}(\pi_{s_{0}}f) .\]
Moreover,
\[ \P(  \mathcal{T}(f)\leq t) 
= \sum_{r=s_{0}}^{s_{1}-1}  \left( \vphantom{\big(}\P\left(  S_{r}(f) \leq t\right)  -  \P\left( S_{r-1}(f) \leq t\right) \right) + \P(  \phi_{s_{0}}(\pi_{s_{0}}f)  \leq t)  ,\] 
and the same is true for $\P_{N}$ as well.
Thus, setting $[\P_{N}-\P](A) = \P_{N}(A)-\P(A)$, we have that
\begin{align}
\label{eq:split.W}
\begin{split}
\W_{1}(F_{N, \mathcal{T}(f)}, F_{ \mathcal{T}(f) } ) 
& \leq \sum_{r=s_{0}}^{s_{1}-1} \int_{\R} \left| \vphantom{\P^{(a)}}   [\P_{N}-\P]( S_{r}(f)\leq t)  - [\P_{N}-\P]( S_{r-1}(f) \leq t) \right| \,dt \\
&\qquad +\int_{\R} \left| [\P_{N}-\P]( \phi_{s_{0}}(\pi_{s_{0}}f )\leq t)  \right| \,dt   \\
& =\mathcal{E}_{1}(f) + \mathcal{E}_{2}(f).
\end{split}
\end{align}

\subsection{Estimating $\mathcal{E}_{1}(f)$}

\begin{lemma}
\label{prop:main.summed.up}
	There is an absolute constant $c$ such that with probability at least $1-\exp(-4\Delta N)$, for every $f\in\mathcal{F}$, 
	\[ \mathcal{E}_{1}(f)
	 \leq 
	 c \left(   d_{\mathcal{F}}  \frac{  \log^{2}(e/\Delta)}{\sqrt N} + \frac{\gamma_{2}(\mathcal{F})}{\sqrt N} \right).\] 
\end{lemma}

The crucial ingredient is a high probability estimate on each increment 
\[\int_{\R} \left| \vphantom{\P^{(a)}}   [\P_{N}-\P]( S_{r} (f)\leq t)  - [\P_{N}-\P]( S_{r-1} (f)\leq t) \right| \,dt :\]

\begin{lemma}
\label{prop:main}
	There is an absolute constant $c$ such that the following holds.
	For every $f\in\mathcal{F}$ and  $r\in\{s_{0},\dots,s_{1}-1\}$, with probability at least $1-\exp(-2^{r+5})$, 
	\begin{align}
	\label{eq:main.prop.statement}
	\begin{split}
	& \int_{\R} \left| \vphantom{\P^{(a)}}   [\P_{N}-\P]( S_{r} (f)\leq t)  - [\P_{N}-\P]( S_{r-1} (f)\leq t) \right| \,dt \\
	& \leq  c \left(  d_{\mathcal{F}}\frac{  \log(e/\Delta )}{\sqrt N} + \|\Delta_{r} f\|_{L_{2}}  \sqrt{\frac{2^{r}}{N}} \right).
	\end{split}
	\end{align}
\end{lemma}

Let us show that Lemma \ref{prop:main.summed.up} is an immediate outcome of Lemma \ref{prop:main}.

\begin{proof}[Proof of Lemma \ref{prop:main.summed.up}]
Observe that
\[|\{ (S_{r}(f),S_{r-1}(f)) : f\in\mathcal{F} \}| \leq \exp(2^{r+4}) .\]
Indeed, recall that  $S_{r}(f) =  \sum_{s=s_{0}}^{r} \phi_{s}(\Delta_{s}f)  + \phi_{s_{0}}(\pi_{s_{0}}f) $ and thus
\begin{align*}
 |\{ S_{r}(f) : f\in\mathcal{F}\}|
&\leq  \sum_{s=s_{0}}^{r}  | \mathcal{F}_{s+1}^{\theta}\times \mathcal{F}_{s}^{\theta}| +  |\mathcal{F}_{s_{0}}^{\theta}|  \\
&\leq    \sum_{s=s_{0}}^{r}   2^{2^{s+1}} 2^{2^{s}}   + 2^{2^{s_{0}}}
\leq \exp(2^{r+3}).
\end{align*}

Now consider $r\in \{s_{0},\dots,s_{1}-1\}$.
By Lemma \ref{prop:main} and the union bound over the pairs $(S_{r}(f),S_{r-1}(f))$, it is evident that with probability at least $1-\exp(-2^{r+4})$  \eqref{eq:main.prop.statement} is true uniformly for every pair $((S_{r}(f),S_{r-1}(f)))_{f\in\mathcal{F}}$.
And, by the union bound over $r$, with probability at least 
\[ 1-\sum_{r\geq s_{0}} \exp(-2^{r+4})
\geq 1-\exp(-2^{s_{0}+2}),\]
for every $f\in\mathcal{F}$, 
\begin{align*}
 \mathcal{E}_{1}(f)
&\leq c_{1} \sum_{r=s_{0}}^{s_{1}  -1}\left(  \|\Delta_{r} f\|_{L_{2}}  \sqrt{\frac{2^{r}}{N}}  + d_{\mathcal{F}}\frac{  \log (e/\Delta)}{\sqrt N} \right)\\
&\leq c_{1} \left( 4 \gamma_{2}(\mathcal{F}) + (s_{1}+1 -s_{0})d_{\mathcal{F}} \frac{  \log (e/\Delta)}{\sqrt N}  \right).
\end{align*}
To complete the proof, note that 
\[ s_{1}-s_{0} \leq \log_{2}(4N) - \log_{2}(\Delta N) 
\leq c \log \left( \frac{e}{\Delta} \right) .\qedhere\]
\end{proof}

Next, fix $f\in\mathcal{F}$ and $r\geq s_{0}$, and set
\[Z =\int_{\R} \left|  \vphantom{\P^{(a)}}  [\P_{N}-\P]( S_{r}(f) \leq t)  - [\P_{N}-\P]( S_{r-1}(f) \leq t) \right| \,dt .\]

We turn to the `main event':
showing that $Z$ concentrates around its expectation, and that the expectation is sufficiently small.

\subsubsection{Controlling $\E Z$}

\begin{lemma}
\label{lem:W1.exp.bounded}
	Let $m>0$ and  $M\geq 1$.
	If $h$ satisfies that  $\|h\|_{L_{2}}\leq m$ and $\|h\|_{L_{\infty}}\leq m M$, then
	\[  \E \W_{1}(F_{N ,h }, F_{h } )  
\leq 2m \frac{\log(eM)  }{\sqrt N}.\]
\end{lemma}
\begin{proof}
	Set $w=h/m$, thus $F_{N,h}(t) = F_{N,w}(t/m)$ and $F_{h}(t) = F_{w}(t/m)$ for $t\in\R$, and $F_{N,w}(t) = F_{ w} (t)\in\{0,1\}$ for $|t|>M$.
	By Fubini's theorem followed by a change of variables,
\begin{align*}
 \E \W_{1}(F_{N,h }, F_{ h } )  
 = \int_{\R} \E |F_{N,h }(s) - F_{ h} (s)| \, ds
= m \int_{-M}^{M} \E |F_{N,w}(t) - F_{ w} (t)| \, dt.
\end{align*}
Clearly, for every $t\in\R$,
\begin{align*}
\left( \E[|F_{N, w}(t) - F_{ w } (t)|^{2} ] \right)^{1/2}
=  \frac{ (F_{w}(t)  (1- F_{ w }(t)) )^{1/2} }{ \sqrt N} .
\end{align*}
Finally, by Markov's inequality, 
\[F_{w}(t)  (1- F_{ w} (t))\leq \|w\|_{L_{2}}^{2} t^{{-2}} \leq  t^{{-2}} ,\]
and  the claim follows because  $\int_{1}^{M} \frac{1}{t} \,dt= \log(M)$.
\end{proof}

\begin{lemma}
\label{lem:W1.Ss.Exp}
	There is an absolute constant $c$ such that for every $f\in\mathcal{F}$ and $r\in\{s_{0}-1,\dots, s_{1}-1\}$,
\[  \E \W_{1}(F_{N,S_{r}(f) }, F_{ S_{r}(f) } )  
\leq  cd_{\mathcal{F}}  \frac{\log(e/\Delta) }{\sqrt N}.\]
In particular,  
\[\E Z \leq  2c d_{\mathcal{F}} \frac{\log(e/\Delta) }{\sqrt N}.\]	
\end{lemma}
\begin{proof}
Observe that  $\| S_{r}(f) \|_{L_{\infty}} \leq  c_{1}  d_{\mathcal{F}} /\sqrt \Delta$.
Indeed, 
\[\|\phi_{s_0}(\pi_{s_{0}}f)\|_{L_{\infty}}  \leq  d_{\mathcal{F}} \sqrt{ \frac{N}{2^{s_{0}}}}\quad\text{and}\quad
\|\phi_{\ell}(\Delta_{\ell}f)\|_{L_{\infty}} \leq  \|\Delta_{\ell}f\|_{L_{2}} \sqrt{\frac{N}{2^{\ell}}} .
\]
Therefore,
\begin{align*}
\begin{split}
\|S_{r}(f) \|_{L_{\infty}} 
&\leq \sum_{\ell=s_{0}}^{r} \|\phi_{\ell}(\Delta_{\ell}f) \|_{L_{\infty}}  +  \|\phi_{s_0}(\pi_{s_{0}}f)\|_{L_{\infty}} \\
&\leq  \sum_{\ell=s_{0}}^{r}  \|\Delta_{\ell} f\|_{L_{2}}\frac{ 2^{\ell/2}}{2^{s_{0}}} \sqrt{N} + d_{\mathcal{F}}\sqrt\frac{N}{2^{s_{0}}} \\ 
&\leq  \frac{4\gamma_{2}(\mathcal{F}) \sqrt N}{ \Delta N} + d_{\mathcal{F}}\sqrt\frac{1}{\Delta } 
\leq 5d_{\mathcal{F}} \sqrt\frac{1}{\Delta } ,
\end{split}
\end{align*}
where we used that $2^{s_{0}}\geq \Delta N$
and that  $ \gamma_{2}^{2}(\mathcal{F}) \leq d_{\mathcal{F}}^{2} \Delta N $. 

Next, note that $\|S_{r}(f) \|_{L_{2}}\leq c_2d_{\mathcal{F}} $.
This holds because  $\| \phi_{s_0}(\pi_{s_{0}}f)\|_{L_{2}}  \leq \| \pi_{s_{0}}f\|_{L_{2}}$ and  $\|\phi_{\ell}(\Delta_{\ell}f) \|_{L_{2}}\leq  \|\Delta_{\ell}f\|_{L_{2}}$; hence 
\begin{align*}
\begin{split}
\|S_{r}(f) \|_{L_{2}}
&\leq \sum_{\ell=s_{0}}^{r} \|\phi_{\ell}(\Delta_{\ell}f) \|_{L_{2}}+ \| \phi_{s_0}(\pi_{s_{0}}f)\|_{L_{2}} \\
&\leq   \frac{4\gamma_{2}(\mathcal{F}) }{ 2^{s_{0}/2}} + d_{\mathcal{F}}.
\end{split}
\end{align*}
Recalling that $2^{s_{0}}\geq\Delta N \geq \gamma_{2}^{2}(\mathcal{F})/d_{\mathcal{F}}^{2}$, it is evident that $\|S_{r}(f) \|_{L_{2}}\leq 5d_{\mathcal{F}}$.

The first part of the lemma follows from Lemma \ref{lem:W1.exp.bounded} and the second one since
\[\E Z \leq \E \W_{1}(F_{N,S_{r}(f) }, F_{ S_{r}(f) } )  + \E \W_{1}(F_{N,S_{r-1}(f) }, F_{ S_{r-1}(f) } ).\qedhere\]
\end{proof}

\subsubsection{Concentration around $\E Z$}

The argument is based on the generalization of the Efron-Stein inequality from \cite{boucheron2005moment}.

Set  $X_{1}',\dots,X_{N}'$ to be independent copies of $X$ that are independent of $X_{1},\dots,X_{N}$, and let
\[\P^{(j)}_{N} = \frac{1}{N} \left( \delta_{X_{1}}+ \cdots \delta_{X_{j-1}}+\delta_{X_{j}'} + \delta_{X_{j+1}}+\cdots \delta_{X_{N}}\right).\]
Denote
\[ Z^{(j)} = \int_{\R} \left| [\P^{(j)}_{N}-\P]( S_{r}(f) \leq t)  - [\P^{(j)}_{N}-\P]( S_{r-1}(f) \leq t) \right| \,dt \]
and put 
\[V=\sum_{j=1}^{N} \E\left[ \left(Z-Z^{(j)} \right)^{2} \mid (X_{i})_{i=1}^{N} \right].\]
Here, the conditional expectation  means that the expectation is taken only with respect to the  $X_j'$.

Using this notation, a direct consequence of Theorem 2 in \cite{boucheron2005moment} is the following:

\begin{theorem}[{\cite[Theorem 2]{boucheron2005moment}}]
\label{thm:BBLM}
For every $q\geq 2$,
\[ \|Z - \E Z\|_{L_{q}}
\leq 6 \sqrt{q} \|\sqrt V\|_{L_{q}}.\]
\end{theorem}

The proof of Lemma \ref{prop:main} is based on Theorem \ref{thm:BBLM} for $q=2^{r+5}$: by Markov's inequality,
\[ \P\left(|Z-\E Z | \geq e 6 \sqrt{q} \|\sqrt V\|_{L_{q}} \right) 
\leq \frac{1}{e^{q}} = \exp(-2^{r+5}).\]

\begin{remark}
In light of the wanted chaining bound, we need to show that 
\[ \|\sqrt V\|_{L_{q}} \leq c \|\Delta_{r} f\|_{L_{2}}/\sqrt{N}\]
 for that choice of $q$.
Such an estimate implies that  $\|\sqrt V\|_{L_{q}}\sim \|\sqrt V\|_{L_{1}}$, and as the proof reveals, this stability is true only up $q\sim 2^{r}$; for higher moments $\|\sqrt V\|_{L_{q}}$ grows with $q$.
\end{remark}

\begin{lemma}
\label{lem:control.V.1}
For every $q\geq 2$, 
\[\|\sqrt V\|_{L_{q}}
\leq \frac{2 \|\Delta_{r}f \|_{L_{2}}}{\sqrt N}  +  \frac{2}{N}  \left\|  \sqrt{ \sum_{j=1}^{N}  \phi_{r}^{2}(\Delta_{r}f(X_{j}))  }\right\|_{L_{q}}.\]
\end{lemma}
\begin{proof}
The expectations appearing in $Z$ and $Z^{(j)}$ are the same, and by the triangle inequality
\begin{align*}
&|Z-Z^{(j)}|\\
&\leq \int_{\R} \left| [\P_{N} - \P_{N}^{(j)}] (S_{r}(f)\leq t ) +  [\P_{N} - \P_{N}^{(j)}] (S_{r-1}(f)\leq t )   \right| \,dt\\
&=  \frac{1}{N}\int_{\R} \left| \left( 1_{\{ S_{r}(f)(X_{j})\leq t \}}-  1_{ \{S_{r}(f)(X_{j}')\leq t \} } \right) - \left( 1_{ \{ S_{r-1}(f)(X_{j})\leq t \} }   - 1_{ \{ S_{r-1}(f)(X_{j}')\leq t \}}   \right) \right| \,dt.
\end{align*}

Next, $S_{r}(f)  = S_{r-1}(f)  + \phi_{r}(\Delta_{r }f )$; therefore,
\[ \{S_{r}(f) \leq t\} = \{ S_{r-1}(f)  + \phi_{r}(\Delta_{r }f )\leq t\} , \]
and the Lebesgue measure of the set 
\[ \left\{ t\in\R : 1_{ \{ S_{r}(f)(X_{j})\leq t \} } - 1_{ \{ S_{r-1}(f)(X_{j})\leq t \} }  \neq 0 \right\}\]
is $|\phi_{r}(\Delta_{r}f(X_{j}))|$.
Using an identical argument for  $X_{j}'$ in place of $X_j$, it is evident that
\[ |Z-Z^{(j)}| \leq  \frac{1}{N} \left(  |\phi_{r}(\Delta_{r}f(X_{j}))|  +   |\phi_{r}(\Delta_{r}f(X_{j}'))| \right),\]
and as $(a+b)^{2}\leq 2 (a^{2} + b^{2})$ for $a,b\in\R$, and $\| \phi_{r}(\Delta_{r}f)\|_{L_{2}}\leq \|\Delta_{r}f\|_{L_{2}}$, we have
\[ \E\left[ \left(Z-Z^{(j)} \right)^{2} \mid (X_{i})_{i=1}^{N}\right] 
\leq  \frac{2}{N^{2}} \left(  \phi_{r}^{2}(\Delta_{r}f(X_{j}))  +   \|\Delta_{r}f \|_{L_{2}}^{2} \right) . \]
The claim now follows from sub-additivity of $a\mapsto \sqrt{a}$.
\end{proof}

With Lemma \ref{lem:control.V.1} in mind, all that remains is to establish a satisfactory estimate on 
\[\left\|  \sqrt{ \sum_{j=1}^{N}  \phi_{r}^{2}(\Delta_{r}f(X_{j}))  }\right\|_{L_{q}}:\]
To that end, we use the following bound on moments of sums of independent random variables established in \cite{latala1997estimation}:

\begin{theorem}[{\cite[Corollary 1]{latala1997estimation}}]
\label{thm:latala}
Let $Y$ be a nonnegative random variable and set $Y_{i},\dots, Y_{N}$ to be independent copies of $Y$.
Then, for every $q\geq 1$,
\[ \left\|  \sum_{j=1}^{N}  Y_{j} \right\|_{L_{q}} 
\sim \sup\left\{ \frac{q}{p} \left( \frac{ N}{q} \right)^{1/p}  \| Y  \|_{L_{p}} : \max\left\{1,\frac{q}{N}\right\} \leq p \leq q \right\}.\]
\end{theorem}

We also need the standard observation that if $\|Y\|_{L_{\infty}}\leq M$ and $\|Y\|_{L_{1}} \leq m$, then for every $p\geq 1$, 
\begin{align}
\label{eq:moment.growth}
\|Y\|_{L_{p}} \leq M \left(\frac{m}{M}\right)^{1/p}.
\end{align}
Indeed,  $|Y|^p = |Y|^{p-1} |Y| \le M^{p-1} |Y|$ and  the claim follows by taking expectations.

\begin{lemma}
\label{lem:control.V.2}
	There is an absolute constant $c$ and for  $q=2^{r+5}$,
\[ \left\|  \sqrt{  \sum_{j=1}^{N}  \phi_{r}^{2}(\Delta_{r}f(X_{j}))  }\right\|_{L_{q}} 
\leq c  \|\Delta_{r}f \|_{L_{2}}  \sqrt N .\] 
\end{lemma}
\begin{proof}
Thanks to Jensen's inequality, it suffices to show that
\[ 
\left\|   \sum_{j=1}^{N}  \phi_{r}^{2}(\Delta_{r}f(X_{j}))  \right\|_{L_{q}} 
\leq c \|\Delta_{r} f\|_{L_{2}}^{2} N. \]
Set $Y= \phi_{r}^{2}(\Delta_{r}f(X))$ and put $Y_{j} = \phi_{r}^{2}(\Delta_{r}f(X_{j}))$.
By the definition of $\phi_{r}$,
\begin{align*}
 \|  Y  \|_{L_1} 
\leq \|\Delta_{r} f\|_{L_{2}}^{2} \quad\text{and}\quad 
\|Y\|_{L_{\infty}}
 \leq  \|\Delta_{r} f\|_{L_{2}}^{2} \frac{N}{2^{r}},
\end{align*}
and by \eqref{eq:moment.growth},  for any $p\geq 1$,
\[ \|Y\|_{L_{p}} \leq  \| \Delta_{r} f \|_{L_2}^{2}  \frac{N}{2^{r}} \left( \frac{2^{r}}{N}  \right)^{1/p}.\]
In particular, for $q=2^{r+5}$, we have that
\begin{align*}
&\sup\left\{ \frac{q}{p} \left( \frac{ N}{q} \right)^{1/p}  \| Y  \|_{L_{p}} : \max\left\{1,\frac{q}{N}\right\} \leq p \leq q \right\} \\
& \quad\leq 2^{5} \| \Delta_{r}f\|_{L_2}^{2}     \cdot \sup\left\{ \frac{2^{r}}{p} \left( \frac{ N}{2^{r}} \right)^{1/p}  \frac{N}{2^{r}} \left( \frac{2^{r}}{N}  \right)^{1/p} : 1 \leq p \leq 2^{r+5} \right\} \\
&\quad = 2^{5} \| \Delta_{r}f\|_{L_2}^{2}  \cdot N,
\end{align*}
and the claim follows from Theorem \ref{thm:latala}.
\end{proof}

With all the ingredients in place, the proof of  Lemma \ref{prop:main} is evident:

\begin{proof}[Proof of Lemma \ref{prop:main}]
	Set $q=2^{r+5}$.
	By Theorem \ref{thm:BBLM} and Markov's inequality, with probability at least $1-\exp(-2^{r+5})$,
	\[ Z \leq \E Z + c_{1} \sqrt{2^{r}} \|\sqrt V\|_{L_{q}}.\]
	Moreover, by Lemma \ref{lem:control.V.1} and Lemma \ref{lem:control.V.2}, 
	\[ \|\sqrt V\|_{L_{q}} \leq c_{2} \frac{ \|\Delta_{r} f\|_{L_{2}}}{\sqrt N },\] 
	and using Lemma \ref{lem:W1.Ss.Exp}, 
	\[\E Z \leq c_{3} d_{\mathcal{F}}  \frac{ \log(e/\Delta) }{ \sqrt N }.\qedhere\]
\end{proof}

\subsection{Controlling $\mathcal{E}_{2}(f)$}

Turning to the starting point of each chain, we have the following:

\begin{lemma}
\label{prop:start.of.chain}
	There is an absolute constant $c$ such that with probability at least $1-\exp(-4\Delta N)$, for every $f\in\mathcal{F}$,
\[ \mathcal{E}_{2}(f) \leq c d_{\mathcal{F}} \left(   \frac{\log(1/\Delta )}{\sqrt N} +\sqrt\Delta \right). \]
\end{lemma}

Lemma \ref{prop:start.of.chain} follows from the same arguments used in the proof of Lemma \ref{prop:main}.
In fact, it is slightly simpler, as no chaining is required.
We therefore omit the details.

\subsection{Putting it all together: proof of Theorem \ref{thm:main}}
	Let $(X_{i})_{i=1}^{N}\in\Omega_{0}$.
	By Lemma \ref{lem:reduction.to.s1},  for every $f\in\mathcal{F}$, 
\[ \W_{1}(F_{N,f}, F_{f} )
\leq  \W_{1}(F_{N,\pi_{s_{1}} f}, F_{\pi_{s_{1}} f} ) + c_{1}B \frac{\theta}{\sqrt N}.\]
	It follows from Lemma \ref{lem:reduction.to.truncation} that for every $f\in\mathcal{F}$,
\begin{align*}
 \W_{1}(F_{N,\pi_{s_{1}}f}, F_{\pi_{s_{1}}f}) 
\leq \W_{1}(F_{N,\mathcal{T}(f)}, F_{\mathcal{T}(f)} ) + c_{2}B^{2}\left( d_{\mathcal{F}} \sqrt\Delta + \frac{\gamma_{2}(\mathcal{F})}{\sqrt N}\right),
\end{align*} 
	and by \eqref{eq:split.W},
	\[\W_{1}(F_{N,\mathcal{T}(f)}, F_{\mathcal{T}(f)} )\leq \mathcal{E}_{1}(f)+\mathcal{E}_{2}(f).\]
	 Using Lemma \ref{prop:main.summed.up}, with $\mu^{\otimes N}$-probability at least $1-\exp(-4\Delta N)$, for every $f\in\mathcal{F}$, 
	\[ \mathcal{E}_{1}(f)
	 \leq 
	 c_{3} \left(   d_{\mathcal{F}}\frac{  \log^{2}(e/\Delta  )}{\sqrt N} + \frac{\gamma_{2}(\mathcal{F})}{\sqrt N} \right),\] 
	 and in a similar fashion, by Lemma \ref{prop:start.of.chain}, with $\mu^{\otimes N}$-probability at least $1-\exp(-4\Delta N)$, for every $f\in\mathcal{F}$, 
	\[ \mathcal{E}_{2}(f)
	 \leq c_{4} d_{\mathcal{F}} \left( \frac{\log(e/\Delta)}{\sqrt N} + \sqrt\Delta \right).\] 
	The proof follows by recalling that  $\Delta N \geq c_5(B) \max\{\theta^{2} / d_{\mathcal{F}}^{2}, \log^{4}(N)\}$  and that $\theta\geq c_6 \gamma_{2}(\mathcal{F})$.
\qed

We end this section with the proof of the optimality in the gaussian case.

\subsection{Proof of Lemma \ref{lem:optimality}}

By  \eqref{eq:W1.rep.Lip} and choosing the 1-Lipschitz functions $\varphi = \pm {\rm id}$, we have that
\[ \W_1(F_{N,x} ,  F_{x})
 \geq \left | \frac{1}{N}\sum_{i=1}^{N} \langle G_{i},x\rangle - \E \langle G,x\rangle \right|
=\mathcal{G}(x). \]

The first claim in Lemma \ref{lem:optimality} follows  from a standard gaussian lower tail bound:
setting $g$ to be the standard gaussian random variable, for every $x\in S^{d-1}$,  $\mathcal{G}(x)$ has the same distribution as $g/\sqrt N$.

For the second claim,  note that  $\sup_{x\in A} \mathcal{G}(x)$ has the same distribution as
\[ \frac{1}{\sqrt N}  \sup_{x\in A } \langle G, x\rangle.\]
By gaussian concentration (see, e.g., Theorem 4.7 in  \cite{pisier1999volume}), for every $u \geq 0$,  with probability at least $1-2\exp(-\frac{1}{4}u^{2})$,
\[\sup_{x\in A } \langle G, x\rangle \geq \E \sup_{x\in A } \langle G, x\rangle - u,\]
which completes the proof.\qed

\section{Proof of Theorems \ref{thm:main.intro} and \ref{thm:main.structure} and optimality}
\label{sec:proof.structure}

The  proofs of Theorem \ref{thm:main.intro} and Theorem \ref{thm:main.structure} follow immediately from Corollary \ref{cor:W1.subgauss} and Theorem \ref{thm:main}, respectively, and the following observation (applied with $p=2$), recalling that  $q_i(f) = F_f^{-1}(\frac{i}{N+1})$. 

\begin{lemma}
\label{lem:quantile.shift}
	For every function $f$, every $p\geq 1$, and every realization of $(X_i)_{i=1}^N$,
	\[ \left| \W_1(F_{N,f} , F_f) - \frac{1}{N} \sum_{i=1}^N \left| f^\sharp(X_i) - q_i(f) \right| \right|
	\leq 10 \frac{ \|f\|_{L_p} }{ N^{1-1/p} }.\]
\end{lemma}
\begin{proof}
	Observe that $F_{N,f}^{-1}(u) = f^{\sharp}(X_{i}) $ for $u\in [\frac{i-1}{N}, \frac{i}{N})$ and therefore,  by \eqref{eq:W1.rep}, 
 	\[  \mathcal{W}_{1}(F_{N,f},F_{f})
	=   \int_{0}^{1}  \left| F_{N,f}^{-1}(u) - F_{f}^{-1}(u)\right|\,du
	=\sum_{i=1}^{N}   \int_{\frac{i-1}{N}}^{\frac{i}{N}}  \left| f^\sharp(X_i) - F_{f}^{-1}(u)\right|\,du.\]
	Hence, to complete the proof it suffices to show that
	\[ (\ast) = \sum_{i=1}^{N}   \int_{\frac{i-1}{N}}^{\frac{i}{N}}  \left|  q_i(f)  - F_{f}^{-1}(u)\right|\,du
	\leq 10 \frac{ \|f\|_{L_p} }{  N^{1-1/p}}.\]
	
	To that end, fix $2 \leq i \leq N-1$.
	Since $\frac{i}{N+1}\in [\frac{i-1}{N},\frac{i}{N})$, the monotonicity of $F_{f}^{-1}$  implies that for every $u\in [\frac{i-1}{N},\frac{i}{N})$,
	\[\left|  q_i(f)  - F_{f}^{-1}(u)\right|
	\leq F_{f}^{-1}\left(\frac{i}{N}\right) - F_{f}^{-1}\left(\frac{i-1}{N}\right) , \]
	and therefore, 
	\[  \sum_{i=2}^{N-1}   \int_{\frac{i-1}{N}}^{\frac{i}{N}}  \left|  q_i(f)  - F_{f}^{-1}(u)\right|\,du 
	\leq \frac{1}{N} \left( F_{f}^{-1}\left(1-\frac{1}{N}\right) - F_{f}^{-1}\left(\frac{1}{N}\right) \right).\]
	
	Set $M= \|f\|_{L_p} N^{1/p}$. 
	By  Markov's inequality  $\P( |f(X)| \geq 2M) < \frac{1}{N}$ and hence  $|F_{f}^{-1}(\frac{1}{N})|\leq 2M$;  similarly the terms $|F_{f}^{-1}(1-\frac{1}{N})|$, $|q_1(f)|$, $|q_N(f)|$ are all bounded by $2M$.
	Moreover, by H\"older's inequality, 
	\[ \int_{0}^{\frac{1}{N}}  |F_{f}^{-1}(u)|\,du
	\leq \left( \int_{0}^{1} 1_{[0,\frac{1}{N}]}(u)\,du\right)^{1-1/p} \|f\|_{L_p}
	=\frac{ M }{N},\]
	and the analogue estimate holds for  $\int_{1-\frac{1}{N}}^{1}  |F_{f}^{-1}(u)|\,du$.
	It follows  that  $(\ast) \leq 10\frac{M}{N}$, as claimed.
\end{proof}

Finally, we present the claims made in the introduction on the optimality of Theorem \ref{thm:main.intro}.
Here $q_i = F_g^{-1}(\frac{i}{N+1})$ and $F_g$ is the standard gaussian distribution function.

\begin{lemma}
\label{lem:optimality.matrix}
	There are absolute constants $c_{1},c_2,c_3$ for which the following holds.
	Let $X=G$ be the standard gaussian vector in $\R^d$, let $A\subset S^{d-1}$ be a symmetric set and let $N\geq c_1$.
	Then for every $\Delta \geq \frac{1}{N}$ and $x\in A$, with probability at least $\exp(-c_{2}\Delta N)$,
\begin{align*}
 \frac{1}{N}\sum_{i=1}^N \left|\inr{X_i,x}^\sharp - q_i\right| 
\geq  c_{3}\sqrt\Delta.
\end{align*}
Moreover, with probability at least $0.99$,
\begin{align*}
\sup_{x\in A}  \frac{1}{N}\sum_{i=1}^N \left|\inr{X_i,x}^\sharp - q_i\right| 
	\geq c_3\frac{\E \sup_{x\in A} \inr{G,x}}{\sqrt N}.
\end{align*}
\end{lemma}
\begin{proof}
	An application of  Lemma \ref{lem:quantile.shift} (with $p=4$) shows that  for every $x\in A$,
	\begin{align}
	\label{eq:matrix.lower.proof}
	 \frac{1}{N} \sum_{i=1}^N \left|\inr{X_i,x}^\sharp - q_i\right| 
	\geq  \W_1(F_{N,x} , F_x) - \frac{c_1}{N^{3/4}}.
	\end{align}
	Both claims follow from Lemma~\ref{lem:optimality}.
	Indeed, note that $\sqrt{\Delta} \geq \frac{1}{\sqrt N}$, and by symmetry of $A$ we have
$\E \sup_{x\in A} \langle G, x\rangle \geq \E |\inr{G,x_0}|=\sqrt{2/\pi}$ for any $x_0\in A$.
	Consequently, the error term $c_1 N^{-3/4}$ in \eqref{eq:matrix.lower.proof} can be absorbed into $\sqrt{\Delta}$ and $\frac{1}{\sqrt{N}} \E \sup_{x\in A} \langle G, x\rangle$, respectively, provided that $N \geq c_2$.
\end{proof}

\appendix

\section{An application of Theorem \ref{thm:main}}
\label{sec:app.application}

Theorem \ref{thm:main}, combined with the Lipschitz-representation of $\W_1$ (see \eqref{eq:W1.rep.Lip}) has an immediate outcome.
To formulate it, let $\mathcal{L}$ denote the set of all $1$-Lipschitz functions from $\R$ to $\R$.

\begin{theorem}
\label{thm:lipschitz}
    In the event where \eqref{eq:main.thm.outcome} holds, we have that
    \[ \sup_{\varphi \in \mathcal{L}}\,\sup_{f\in\mathcal{F}} 
    \left( \frac{1}{N}\sum_{i=1}^N \varphi(f(X_i)) - \E \varphi(f(X)) \right) \leq d_\mathcal{F}\sqrt\Delta.\]
\end{theorem}

Two consequences of Theorem~\ref{thm:lipschitz} are worth mentioning.

First, assume that $(f(X))_{f\in\mathcal{F}}$ is a gaussian process and that
$\E\sup_{f\in\mathcal{F}} G_f \geq \log^2 N$.
Then Theorem~\ref{thm:lipschitz} yields the following uniform version of the contraction principle (see, e.g., \cite{ledoux1991probability} for the classical result):
\[
\E \sup_{\varphi \in \mathcal{L}}\,\sup_{f\in\mathcal{F}} 
    \left( \frac{1}{N}\sum_{i=1}^N \varphi(f(X_i)) - \E \varphi(f(X)) \right)
    \sim
 \E \sup_{f\in\mathcal{F}}
\left| \frac{1}{N}\sum_{i=1}^N f(X_i) - \E f(X) \right|.
\]

A second application of Theorem~\ref{thm:lipschitz} addresses a central challenge in empirical process theory:  establishing gaussian concentration in non-gaussian settings and under minimal structural assumptions on the indexing class.
To that end, suppose that  $\mathcal{F}$ satisfies the one-sided coarse estimate of Assumption~\ref{ass:main} with
$\theta \sim \E \sup_{f\in\mathcal{F}} G_f$ and $\theta \geq \log^2 (N)$.
Applying Theorem~\ref{thm:lipschitz} with $\varphi=\pm \mathrm{id}$ yields that, with high probability,
\[
\sup_{f\in\mathcal{F}}
\left| \frac{1}{N}\sum_{i=1}^N f(X_i) - \E f(X) \right|
\leq
d_{\mathcal{F}}\,\frac{\E\sup_{f\in\mathcal{F}} G_f}{\sqrt{N}}.
\]
Thus, a `weak' one-sided structural condition leads to a gaussian estimate on the empirical process.

\section{On Assumption \ref{ass:main}}
\label{sec:app.proof.ass}

%

\subsection{Proof of  Lemma \ref{prop:assumption.gaussian}}

	Recall that $d_{\mathcal{F}}=\sup_{f\in\mathcal{F}}\|f\|_{L_2}$ and $d^\ast(\mathcal{F})\sim \gamma_2^2(\mathcal{F}) / d_\mathcal{F}^2$, and set
	\[\mathcal{H} = {\rm conv} \left( \{f- \tilde f : f,\tilde f \in\mathcal{F}\}\right) .\]
	Clearly  $\mathcal{F}\subset \mathcal{H}$,  $0\in\mathcal{H}$, and $d_\mathcal{H}\leq 2 d_{\mathcal{F}}$.
	Moreover,  a standard consequence of Talagrand's  majorizing measures theorem is that $\gamma_{2}(\mathcal{H})\sim\gamma_{2}(\mathcal{F})$; thus $d^\ast(\mathcal{H})\sim d^\ast(\mathcal{F})$.
	
	Let $r\geq   \frac{ \gamma_2(\mathcal{F}) }{ \sqrt N}$, and put 
	\[\mathcal{H}_r = \{ h\in \mathcal{H} : \|h\|_{L_{2}}\leq r\}.\]	
	It suffices to show that, with exponentially high probability, $\sup_{h\in\mathcal{H}_r} \|h\|_{L_2(\P_N)} \leq c_1L r$.
	If that is true, the lemma follows because $\mathcal{H}$ is convex and contains 0 and therefore is star-shaped around 0: for $h\in\mathcal{H}$ and $u\in[0,1]$, $uh\in\mathcal{H}$.
	
	To that end, an application of Theorem 1.13 in \cite{mendelson2016upper} and the preceding discussion therein shows that, for every $\lambda\geq 1$, with probability at least $1-2\exp(- c_2\lambda^2 d^{\ast}( \mathcal{H}_r))$,
	\[ \sup_{h\in \mathcal{H}_r} \left| \frac{1}{N}\sum_{i=1}^{N} h^{2}(X_{i}) - \|h\|_{L_{2}}^{2} \right|
	\leq c_{3} L^{2}  \left(
	 \lambda d_{\mathcal{H}_r} \frac{ \gamma_{2}(\mathcal{H}_r) }{\sqrt N} +  \lambda^2 \frac{ \gamma_{2}^{2}(\mathcal{H}_r) }{ N}
	\right)
	=\mathcal{E}_\lambda.\]
	Clearly $d_{\mathcal{H}_r}\leq  r $ and  $\gamma_{2}(\mathcal{H}_r) \leq  \gamma_{2}(\mathcal{H})\sim\gamma_2(\mathcal{F})$.
	Setting $\lambda = c_4\frac{r \sqrt N }{\gamma_2(\mathcal{H}_r)}$ we have that $\lambda\geq 1$.
	Moreover, $\mathcal{E}_\lambda\leq c_{5}L^{2 }r^{2}$ and using that $d_{\mathcal{H}_r} \sim \min\{r, d_{\mathcal{F}} \}$,
	\[ \lambda^2 d^{\ast}( \mathcal{H}_r) 
	=   c_4^2  \frac{r^2 N }{ d_{\mathcal{H}_r}^2} 
	\sim \max\left\{ N, \frac{r^2 N }{ d_{\mathcal{F}}^2} \right\}  
 \] 
	from which the claim follows.
\qed

We proceed to show that Assumption \ref{ass:main} is satisfied for other, heavy-tailed classes $\mathcal{F}$ as well, starting with:

\subsection{Classes of linear functionals indexed by $S^{d-1}$}

Let $\mathcal{F}=\mathcal{F}_{S^{d-1}} =  \{ \langle \,\cdot \,, x\rangle : x\in S^{d-1}\}$ be  the class of linear functionals indexed by the entire sphere, and let  $X\in\R^{d}$ be a zero mean random vector that satisfies $L_{q}-L_{2}$ norm-equivalence with constant $L$: for every $x\in \R^{d}$,
\[  \| \langle X, x\rangle \|_{L_{q}}
\leq L  \| \langle X, x\rangle \|_{L_{2}} . \]
If $X$ is isotropic and  $\|X\|_{2}$ is a well-behaved random variable, under such a norm-equivalence the random matrix $\frac{1}{\sqrt N} \sum_{i=1}^{N}\inr{X_{i},\cdot}e _{i}$ satisfies the quantitative Bai-Yin asymptotics \cite{bai1993limit}; namely, that
\begin{align*}
\zeta_{d,N}=\sup_{x\in S^{d-1}} \left| \frac{1}{N}\sum_{i=1}^{N} \langle X_{i},x\rangle^{2} - \E \langle X,x\rangle^{2} \right|
&\lesssim \sqrt\frac{d}{N}.
\end{align*}

The  following result is an immediate outcome of  \cite[Corollary 2]{tikhomirov2018sample} and  is the current state of the art on such Bai-Yin type estimates.

\begin{theorem}
\label{thm:Bai.Yin}
For every $L,R\geq 1$, $q>4$ there is a constant $c=c(q,L)$ such that the following holds.
If $X$ is isotropic and satisfies $L_{q}-L_{2}$ norm-equivalence with constant $L$,  $N\geq 2d$, and  $\max_{i\leq N} \|X_{i}\|_{2} \leq R (dN)^{1/4}$ with probability $1-\delta$, then, with probability at least $1-\frac{1}{d}-\delta$, 
\[\zeta_{d,N}\leq cR^2  \sqrt\frac{d}{N}.\]
\end{theorem}

Note that since $\E\|X\|_{2}^{2} =  d$ and $N\geq d$, the condition that $\max_{i\leq N }  \|X_i\|_2\leq R(dN)^{1/4}$ is naturally satisfied, with constants $R$ and $\delta$ depending on the tail decay exhibited by $\|X\|_2$.
In fact, there are versions of Theorem \ref{thm:Bai.Yin} that hold with higher probability, but we will not pursue that aspect here.

A direct consequence of Theorem \ref{thm:Bai.Yin} is that:

\begin{corollary}
\label{cor:Bai.Yin}
	Let $\Omega_{0}$ be the event in which the assertion of Theorem \ref{thm:Bai.Yin} holds  and assume that  $  N\geq  16R^4c^2 d $.
	Then Assumption \ref{ass:main} is satisfied with $B=2$ and $\theta=0$.
\end{corollary}

A more interesting question is whether a dimension-free analogue to Corollary \ref{cor:Bai.Yin} is true for non-isotropic $X$, with a condition on $N$ in terms of the covariance matrix of $X$ rather than the assumption that $N\gtrsim d$.
While this does not follow immediately from the dimension-free Bai-Yin theorem recently established in \cite{abdalla2022covariance} (see also \cite{jirak2025concentration}), such an estimate happens to be true.
To  formulate the result, set $\Sigma=\C[X]$, let ${\rm tr}(\Sigma)$ be  its trace, denote by $\|\Sigma\|_{\rm op}$ its operator norm, and put
\[ {\rm r}(\Sigma) = \frac{ {\rm tr}(\Sigma)}{\|\Sigma\|_{\rm op}},\]
i.e., the effective rank of $\Sigma$.
It is standard to verify that  ${\rm tr}(\Sigma) \sim \gamma_2^2(S^{d-1})$ and  that  ${\rm r}(\Sigma)\sim d^\ast(S^{d-1})$.
Moreover, in the context of establishing Assumption \ref{ass:main},  one may assume without loss of generality that $\Sigma$ is diagonal with decreasing, strictly positive eigenvalues $\lambda_i =\lambda_i(\Sigma)$ for $i=1,\dots,d$.
We denote by $P_{1:N}\colon\R^{d}\to\R^{d}$ the $\ell_2$-projection onto ${\rm span}(e_1,\dots,e_{\min\{d,N\}})$, and since $\Sigma$ is diagonal, $P_{1:N}$ is also the $L_2$-projection.

\begin{lemma}
\label{lem:Bai.Yin.non.isotropic}
	For every  $L,R\geq 1$ and $q>4$ there are constants $c_1,c_2,c_3$ depending on $q$ and $L$ such that the following holds.
	Assume that $X$ satisfies $L_q-L_2$ norm-equivalence with constant $L$ and that 
	with probability $1-\delta$,
	\begin{itemize}
	\item[(i)]  $\max_{i\leq N} \| P_{1:N} \Sigma^{-1/2}X_i \|_2 \leq  R \sqrt N$,
	\item[(ii)] $\max_{i\leq N}\|X_i\|_2\leq R \sqrt{{\rm tr}(\Sigma)}$.
	\end{itemize}
	If $N\geq c_1 ( {\rm r}(\Sigma) + R^4)$,
	then with  probability at least $1-\frac{c_2}{N}-2\delta$, for every $x\in \R^d$ satisfying $\|x\|_2\leq 1$,
	\begin{align}
	\label{eq:Bai.Yin.free}
	\|\inr{X,x}\|_{L_2(\P_N)}
	\leq c_3R \left( \|\inr{X,x}\|_{L_2} + \sqrt \frac{{\rm tr}(\Sigma)}{N} \right) .
	\end{align}
\end{lemma}

To put  conditions (i) and (ii) into some perspective, note that
$\E\|X\|_{2}^{2} = {\rm tr}(\Sigma)$ and 
$\E \| P_{1:N} \Sigma^{-1/2}X \|_2^{2} =  N$; hence (i) and (ii) are  naturally satisfied with constants $\delta$ and $R$, depending on the tail behaviour of  $X$ and $\Sigma^{-1/2} X$.
That  dependence is the reason why we kept $R$ explicit in Theorem \ref{lem:Bai.Yin.non.isotropic}.

\begin{corollary}
\label{cor:Bai.Yin.free}
	Let $\Omega_{0}$ be the event in which \eqref{eq:Bai.Yin.free} holds; then Assumption \ref{ass:main} is satisfied with $B=c_3 R $ and $\theta=c_3 R \sqrt{{\rm tr}(\Sigma)}$.
	In particular, if 
	\[ \Delta N \geq c_1 \max\left\{ {\rm r}(\Sigma) , \log^4(N)\right\}, \]
	then with probability at least $1-\frac{c_2}{N}-2\delta - \exp(-c_3\Delta N)$,
	\[  \sup_{x\in S^{d-1}} \W_1(F_{N,x},F_x)
	\leq \sqrt{\lambda_1 \Delta}.\]
\end{corollary}

\begin{proof}[Proof of Lemma \ref{lem:Bai.Yin.non.isotropic}]

Let  $\beta$  be the absolute constant  appearing in Theorem \ref{thm:AZ} below and set $\eta=\frac{1}{2\beta} \leq 1 $.
Put $K=\eta N$ and assume without loss of generality that $K$ is an integer.
Set  $L={\rm span}(e_1,\dots,e_{K})$ and $M={\rm span}(e_{K+1},\dots, e_{d})$ (if $K\geq d$, then $M=\emptyset$); $P_{L}$ and $P_{M}$ are the $\ell_2$-projections on $L$ and $M$, respectively.

The estimates in case $x\in L$ and $x\in M$ are different: we show that 
with probability at least $1-\frac{1}{K}-\delta$, for every $x\in L$, 
\begin{align}
\label{eq:ass.estim.L}
\|\inr{X,x}\|_{L_2(\P_N)} \leq c_1(q,L)R \|\inr{X,x}\|_{L_2}
\end{align}
and with probability at least $1-c_2(q,L)\frac{1}{N}-\delta$,  
\begin{align}
\label{eq:ass.estim.M}
\sup_{x \in M : \|x\|_{2}\leq 1}  \,\sum_{i=1}^N\inr{X_i,x}^2\leq c_3(q,L)R^2 {\rm tr}(\Sigma).
\end{align}

\vspace{0.5em}
\noindent
\emph{Step 1 (the estimate on the subspace $L$):}
Set $Y_i=P_L \Sigma^{-1/2}X_i$ and recall that with probability $1-\delta$,  
\[\max_{i\leq N} \|Y_i\|_{2}
\leq  \max_{i\leq N} \|P_{1:N} \Sigma^{-1/2} X_i\|_{2}
\leq R \sqrt N
= R \eta^{-1/4} (KN)^{1/4}.\] 
It follows from Theorem \ref{thm:Bai.Yin} that with probability at least $1-\frac{1}{K}-\delta$,
\begin{align*}
 \sup_{x\in L : \|x\|_{2}=1 }\, \left| \frac{1}{N}\sum_{i=1}^N \inr{Y_i,x}^2 - 1 \right|
&\leq  c_4(q,L)  R^2  \eta^{-1/2}    \sqrt \frac{K}{N}
= c_4 R^2.
\end{align*}
On that event, for every $x\in L $ satisfying $  \|x\|_{2}=1$,
\[\|\inr{Y,x}\|_{L_2(\P_N)}^2 
\leq  c_4 R^2 +1 
= (c_4 R^2 + 1 ) \|\inr{Y,x}\|_{L_2}^2,\]
 and positive homogeneity shows that  for every $x\in L$, $\|\inr{Y,x}\|_{L_2(\P_N)} \leq \sqrt{c_4 R^2 + 1} \|\inr{Y,x}\|_{L_2}$.
Hence \eqref{eq:ass.estim.L} follows by transforming back to the original coordinate system.

\vspace{0.5em}
\noindent
\emph{Step 2 (the estimate on $M$):}
Set $Z=P_{M}X$ and $Z_i=P_{M}X_{i}$, and put $\Sigma^Z=\C[Z]$.
Observe that, since $(\lambda_i)_{i=1}^d$ is non-increasing,
\[ \|\Sigma^{Z}\|_{\rm op}
=\lambda_{K+1}
\leq \frac{1}{K} \sum_{i=1}^{K} \lambda_i
\leq \frac{{\rm tr}(\Sigma) }{K}.\]
Assume first that  $\lambda_{K+1} \geq \frac{1}{2} \frac{{\rm tr}(\Sigma) }{K}$; we explain the minor changes needed in the general case at the end of the proof.

By interchanging  two suprema,
\begin{align*} 
\sup_{x\in M : \|x\|_{2}\leq 1}\, \sum_{i=1}^N \inr{Z_i,x}^2
&=  \sup_{x\in M : \|x\|_{2}\leq 1} \sup_{a\in S^{N-1}}  \left( \sum_{i=1}^N a_i\inr{Z_i,x} \right)^2 \\
&= \sup_{a\in S^{N-1}}  \left\| \sum_{i=1}^N a_i Z_i \right\|_2^2
= (\ast).
\end{align*}
Following the notation used in \cite{abdalla2022covariance} and setting  $[N]=\{1,\dots,N\}$, let 
\[ f(s, [N]) 
=\sup_{a\in S^{N-1} \, : \, |\{ i \leq N : a_i \neq 0\}| \leq s} \,  \left\| \sum_{i=1}^N a_i Z_i \right\|_2^2\]
for $s\geq 1$.
Note that $(\ast)= f(N,[N])$ and that $f(N,[N])\leq \frac{1}{\eta^2} f(K,[N])$; hence to establish \eqref{eq:ass.estim.M} it remains to show that with high probability, $f(K,[N])\lesssim R^2 {\rm tr}(\Sigma)$.
To that end, we apply  \cite[Theorem 3]{abdalla2022covariance}:

\begin{theorem}[{\cite[Theorem 3]{abdalla2022covariance}}]
\label{thm:AZ}
There is an absolute constant $\beta\geq 1$ and constants $c,c'$ depending only on $q$ and $L$ such that the following holds:
For all integers $N$ and $s$ that satisfy   $\frac{1}{\beta}N\geq s\geq {\rm r}(\Sigma^Z)$,  with probability at least $1-\frac{c}{N}$,
\begin{align*}
f(s,[N]) 
\leq c' \left( \max_{i\leq N} \|Z_i\|_2^2 + \|\Sigma^Z\|_{\rm op} \cdot s  \left(\frac{N}{s}\right)^{4/(4+q)} \log^4\left( \frac{N}{s}\right) \right).
\end{align*}
\end{theorem}

We claim that the condition  $\frac{1}{\beta}N\geq s\geq {\rm r}(\Sigma^Z)$ in Theorem \ref{thm:AZ} is satisfied for $s=2K$.
Indeed, since  $\|\Sigma^Z\|_{\rm op}=\lambda_{K+1} \geq \frac{ {\rm tr}(\Sigma) }{2K}$,  we have 
\[ {\rm r}(\Sigma^Z) 
= \frac{{\rm tr}(\Sigma^Z)}{ \|\Sigma^Z\|_{\rm op}}
\leq \frac{{\rm tr}(\Sigma)}{\lambda_{K+1}}
\leq 2K,\]
and the claim follows since  $K=\eta N$ and $\eta=  \frac{1}{2\beta}$.
Moreover, using that $\|\Sigma^Z\|_{\rm op}\leq \frac{ {\rm tr}(\Sigma) }{K}$, that $\|Z_{i}\|_{2}\leq \|X_{i}\|_{2}$, and that $\max_{i\leq N} \|X_i\|_2^2 \leq R^2{\rm tr}(\Sigma)$ with probability $1-\delta$ by assumption, it is straightforward to verify that with probability $1-\delta$,
\[\max_{i\leq N} \|Z_i\|_2^2 + \|\Sigma^Z\|_{\rm op} \cdot 2K  \left(\frac{N}{2K}\right)^{4/(4+q)} \log^4\left( \frac{N}{2K}\right) 
\leq  c_5(q,L) R^2 \mathop{\rm tr}(\Sigma).\]
Finally, $f(K,[N])\leq f(2K,[N])$, which concludes the proof of \eqref{eq:ass.estim.M}.

\vspace{0.5em}
\noindent
\emph{Step 3 (putting the estimates together):}
On the intersection of the events in which the assertions of Step 1 and Step 2 hold, set $x\in\R^d$ with $\|x\|_2\leq 1$.
Write $x=y+z\in L\oplus M$ and note that $\|\inr{X,y}\|_{L_2}\leq \|\inr{X,x}\|_{L_2}$.
Therefore, 
\begin{align*}
 \|\inr{X,x}\|_{L_2(\P_N)}
&\leq \|\inr{X,y}\|_{L_2(\P_N)}  + \|\inr{X,z}\|_{L_2(\P_N)}
\\
&\leq c_6(q,L)R \left(  \|\inr{X,y}\|_{L_2} + \sqrt{\frac{{\rm tr}(\Sigma)}{N}} \right) .
\end{align*}
as claimed.

\vspace{0.5em}
\noindent
\emph{Step 4 (on the assumption that $\lambda_{K+1} \geq \frac{{\rm tr}(\Sigma) }{2K}$):}
If  the assumption is not satisfied, we proceed as follows.
Set $\xi= \frac{{\rm tr}(\Sigma) }{K}$, let $B$ be the standard Bernoulli random vector in $\R^{K}$, independent of $X$, and put  \[\tilde X = (X,\sqrt \xi B)\in\R^{d+K}.\]
Setting  $\tilde \Sigma=\C[\tilde X]$, we have ${\rm tr}(\tilde \Sigma) = 2 {\rm tr}(\Sigma)$ and $\|\tilde \Sigma\|_{\rm op} \geq \|\Sigma\|_{\rm op}$;  thus   ${\rm r}(\tilde \Sigma)\leq 2 {\rm r}(\Sigma)$ and  $\lambda_K(\tilde \Sigma)\geq \xi = \frac{{\rm tr}(\tilde \Sigma)}{2K}$.
Moreover,  $\tilde  X$ satisfies $L_{q}-L_{2}$ norm equivalence with constant $L+c(q)$, and, almost surely, $\|\tilde  X\|_{2}\leq \|X\|_2 + \sqrt{{\rm tr}(\Sigma)}$ and $\|P_{1:N} \tilde \Sigma^{{-1/2}} \tilde  X\|_{2} \leq \|P_{1:N}\Sigma^{{-1/2}} X \|_2 + \sqrt{K}$.
We repeat Step 1-3 for $\tilde X$.
\end{proof}

\subsection{Classes of linear functionals indexed by $A\subset \R^{d}$}

Let  $\mathcal{F}=\mathcal{F}_{A} =  \{ \langle \,\cdot \,, x\rangle : x\in A\}$ where $A\subset \R^d$ and $X\in\R^{d}$ is isotropic. 
Let   $c_{0}$ be a suitable absolute constant, set $w$ to be a random variable with zero mean and unit variance that satisfies the following `local' moment condition---that for some $\alpha>0$,
\begin{align}
\label{eq:local.psi.alpha}
 \|w\|_{L_p} \leq L p^{1/\alpha} \|w\|_{L_{2}} \quad\text{for }  2\leq p \leq c_{0}\log(ed),
\end{align}
and put  $X=(w_{1},\dots,w_{d})$ where the $w_{i}$'s are independent copies of $w$.

\begin{remark}
	Note that  \eqref{eq:local.psi.alpha} is indeed much weaker than \eqref{eq:def.subgaussian}; for example $w$ need not  have any moments beyond $c_{0}\log(ed)$.
\end{remark}
 
If the random vector $X$ satisfies the local moment condition and $A$ is not too small---in the sense that its critical dimension is at least logarithmic in $d$, then Assumption \ref{ass:main} is satisfied:

\begin{lemma}
\label{prop:assumption.non.gaussian}
	There are absolute constants $c_{0}, c_{1},c_{2}$ and, for every $L,\alpha >0$, there is a constant $c_{3}=c_{3}(L,\alpha)$ such that the following holds.
	Suppose that $w$ satisfies \eqref{eq:local.psi.alpha} with $L$ and $c_{0}$,  $A\subset \R^d$ satisfies $0\in A$ and   $d^{\ast}(A)\geq c_{1}\log(ed)$, and $N\geq \log^{4/\alpha+ 1}(ed)$.
	Put
	\[ \xi = \log^{2/\alpha + 1}\left(\frac{ed}{ d^{\ast}(A)}\right).\] 
	Then  Assumption \ref{ass:main} is satisfied with set  $\Omega_{0}$ of measure $\P(\Omega_{0})\geq 1 - \exp(-c_2 d^{\ast}(A))-d^{-10} $ and constants $B=c_{3} $, $\theta = c_3  \xi \gamma_{2}(A) $.

\end{lemma}

\begin{corollary}
\label{prop:assumption.non.gaussian.main.W1}
	In the setting of Lemma \ref{prop:assumption.non.gaussian}  and using its notation: with probability at least $1 - \exp(-c_2 d^{\ast}(A))-d^{-10}$,
	\[ \sup_{x\in A } \W_1(F_{N,x},F_x) 
	\leq c_4(L,\alpha)\frac{ \max\{  \xi \gamma_2(A) \, ,\,  \log^2(N) \} }{\sqrt N}.\]
\end{corollary}


The proof of Lemma \ref{prop:assumption.non.gaussian}  resembles that one of Lemma \ref{prop:assumption.gaussian} but requires some preparations.
Denoting by $\Gamma $ the matrix that has $ X_{i}$ as its rows, the central ingredient in the proof of Lemma \ref{prop:assumption.gaussian} was to establish  appropriate estimates on $\sup_{x }| \|\frac{1}{\sqrt N}\Gamma x\|_{2}^{2} - \|x\|_{2}^{2}|$,  where the supremum is taken over $x$ in localizations of $A$.
To establish those estimates, we will rely on the results in \cite{bartl2022random} on structure preservation of random matrices with i.i.d.\ columns.
Note that since each $X$ has i.i.d.\ coordinates, the columns of $\Gamma $ are i.i.d., and we denote them by $Z_{1},\dots, Z_{d}\in\R^{N}$, thus $\Gamma=\sum_{j=1}^d \inr{\cdot,e_j}Z_j$.

We start by proving the (rather standard) result showing that $Z$ inherits the local moment condition required in \cite{bartl2022random}.

\begin{lemma}
\label{lem:local.psi.alpha}
	There is an absolute constant $c$ such that the following holds.
	Let $w$ be a random variable with zero mean and unit variance and set $Z=(w_{1},\dots, w_{N})$ where the $w_{i}$'s are independent  copies of $w$.
	Then, for every $z\in S^{N-1}$ and every  $p\geq 2$,
	\[\|\inr{Z,z}\|_{L_p}\leq c \sqrt{p} \|w\|_{L_{p}}.\]
	Moreover, if $Z_{1},\dots, Z_{d}$ are independent copies of $Z$, then for every $p\geq 2$ and for every $\lambda \geq 1$,  with probability at least $1-d\lambda^{-p}$,
	\[ \max_{j=1,\dots,d }\left| \frac{ \|Z_{j}\|_{2}^{2} }{N} - 1 \right| 
	\leq  c \lambda   \frac{ \sqrt p \|w^{2}\|_{L_p} }{\sqrt N} .\]
\end{lemma}
\begin{proof}
	Let $(\varepsilon_i)_{i=1}^N$ be i.i.d.\ Bernoulli random variables that are also independent of $Z$.
	By a standard symmetrization argument (see, e.g., \cite[Lemma 6.3]{ledoux1991probability}),
	\[ \|\inr{Z,z}\|_{L_p} 
	=\left\| \sum_{i=1}^N w_{i}  z_i \right\|_{L_p}
	\leq 2 \left\| \sum_{i=1}^N \varepsilon_i w_{i } z_i \right\|_{L_p}.\]
	Moreover, Bernoulli random variables are subgaussian with an absolute constant, in particular $\|\sum_{i=1}^N \varepsilon_i v_i \|_{L_p}\leq c \sqrt{p}\|v\|_2$ for every $v\in\R^N$ and $p\geq 2$.
	Applying the latter inequality  conditionally on $Z$,
	\begin{align*}
	\left\| \sum_{i=1}^N \varepsilon_i w_{i} z_i  \right\|_{L_p}^p
	=\E_Z\E_{\varepsilon}\left| \sum_{i=1}^N \varepsilon_i w_{i} z_{i} \right|^p
	\leq \E_Z  \left(c\sqrt{p}\right)^p  \left| \sum_{i=1}^N  w_{i}^2 z_i^2 \right|^{p/2},
	\end{align*}
	where here $\E_Z$ and $\E_{\varepsilon}$ denote the expectation  with respect to $Z$ and $(\varepsilon_i)_{i=1}^N$, respectively. 
	Let $B_1^N$  be the $\ell_1^N$-unit ball and observe that
	\begin{align*}
	\sup_{z\in B_2^N} \E_Z \left| \sum_{i=1}^N w_{i}^2 z_i^2 \right|^{p/2}
	=\sup_{v\in B_1^N} \E_Z \left| \sum_{i=1}^N w_{i}^2 v_i \right|^{p/2}
	=\max_{v\in B_1^N} \psi(v).
	\end{align*}
	As $\psi$ is convex, it follows that $\max_{v\in B_1^N} \psi(v)$ is attained at an extreme point of $B_1^N$, i.e.\ for $v$ of the form $v=\pm e_i$, $1\leq i\leq N$; hence
	\[\max_{v\in B_1^N} \psi(v)
	=\max_{1\leq i\leq N} \E_Z \left|  w_{i}^2  \right|^{p/2}
	=\|w\|_{L_p}^p.\]
	Combining all estimates, $\| \inr{Z,z}\|_{L_p} \leq  2c \sqrt{p} \|w\|_{L_p}$.
	
	As for the second statement, by a standard symmetrization argument followed by the arguments presented in the previous step, 
	\begin{align*}
	\left\| \frac{\|Z\|_{2}^{2}}{N} - 1 \right\|_{L_{p}}
	&\leq 2 \frac{1}{N} \left\| \sum_{i=1}^N \varepsilon_i w_{i }^{2} \right\|_{L_p} 
	\leq \frac{2c  \sqrt{p} \|w^{2}\|_{L_{p}}}{\sqrt N}.
	\end{align*}
	Hence, by Markov's inequality, for every $\lambda\geq 1$,
	\[ \P\left( \left| \frac{ \|Z\|_{2}^{2} }{N} - 1 \right| \geq \lambda  \frac{2c\sqrt{p} \|w^{2}\|_{L_{p}} }{\sqrt N} \right)\leq \lambda^{-p},\]
	and the claim follows from the union bound over $j=1,\dots, d$.
\end{proof}

\begin{proof}[Proof of Lemma \ref{prop:assumption.non.gaussian}]
	Set $Z = (w_{1},\dots, w_{N})$.
	The first part of Lemma \ref{lem:local.psi.alpha} implies that for every $z\in \R^{N}$ and every  $2\leq p \leq c_{0} \log(d)$, 
\[\|\inr{Z,z}\|_{L_p}\leq c_{1} L p^{1/\alpha + 1/2}   \|z\|_{2}\]
	and the second part of that lemma (applied with $\lambda = e^{11}$ and $p=\log(d)$) shows that,  with probability at least $1-d^{-10}$,
\[ \max_{j=1,\dots,d }\left| \frac{ \|Z_{j}\|_{2}^{2} }{N} - 1 \right| 
	\leq  c_{2} L    \frac{\log^{2/\alpha + 1/2}(d)  }{\sqrt N} .\]
	Thus $Z$ satisfies the assumption needed in \cite{bartl2022random} and it follows from Theorem 1.5 therein that	for every set $V\subset \R^{d}$ satisfying $d^{\ast}(V)\geq \log(d)$, setting
	\[ \mathcal{E}(V) =   \frac {\log^{2/\alpha + 1/2}(d)  }{\sqrt N}  d_{V}^{2}+  \log^{2/\alpha+ 1} \left( \frac{ e d }{ d^{\ast}(V) }\right)  \left( d_{V} \frac{\gamma_{2}(V)}{\sqrt N } +  \frac{\gamma_{2}^{2}(V)}{ N } \right)  ,\]
	with probability at least $1- d^{-10} - 2\exp(-c_{3} d^{\ast}(V) )$,
	\begin{align}
	\label{eq:almost.gaussian.embedding}
	\sup_{x\in V} \left| \left\| \frac{1}{\sqrt N} \Gamma x \right\|_{2}^{2} - \|x\|_{2}^{2 } \right| 
	\leq  c_{4}(L,\alpha) \mathcal{E}(V).
	\end{align}

	From this point, the proof of the lemma follows from the same path as the one presented for Lemma \ref{prop:assumption.gaussian} and we only sketch it:
	Set $\xi = \log^{2/\alpha + 1} ( \frac{ed}{d^{\ast}(A) }),$
	and put $r =  \xi \frac{\gamma_{2}(A)}{\sqrt N}$.
	Apply \eqref{eq:almost.gaussian.embedding} to 
	\[ A_r=\left\{ x\in {\rm conv} (\{ y-z: y,z\in A\cup\{0\}\}) : \|x\|_{2}\leq r\right\}.\]
	The claim follows by 	noting that  $d^{\ast}(A_{r})\geq c_{5}d^{\ast}(A)\geq c_{6} \log(d)$ and that  $\mathcal{E}(A_{r}) \leq c_{7}(L,\alpha)r^{2}$, where the latter holds because  $N \geq \log^{4/\alpha+1}(d)$.
\end{proof}

\vspace{1em}

\noindent
{\bf Acknowledgement:} 
This research was funded in whole or in part by the Austrian Science Fund (FWF) [doi: 10.55776/P34743 and 10.55776/ESP31], the Austrian National Bank [Jubil\"aumsfond, project 18983], and a Presidential Young Professorship grant [`Robust statistical learning for complex data'].

\bibliographystyle{abbrv}

\begin{thebibliography}{10}

\bibitem{abdalla2022covariance}
P.~Abdalla and N.~Zhivotovskiy.
\newblock Covariance estimation: Optimal dimension-free guarantees for
  adversarial corruption and heavy tails.
\newblock {\em Journal of the European Mathematical Society}, 28(4):1809--1847,
  2026.

\bibitem{adamczak2010quantitative}
R.~Adamczak, A.~Litvak, A.~Pajor, and N.~Tomczak-Jaegermann.
\newblock Quantitative estimates of the convergence of the empirical covariance
  matrix in log-concave ensembles.
\newblock {\em Journal of the American Mathematical Society}, 23(2):535--561,
  2010.

\bibitem{bai1993limit}
Z.-D. Bai and Y.-Q. Yin.
\newblock Limit of the smallest eigenvalue of a large dimensional sample
  covariance matrix.
\newblock {\em Annals of Probability}, 21(3):1275--1294, 1993.

\bibitem{bartl2022random}
D.~Bartl and S.~Mendelson.
\newblock Random embeddings with an almost gaussian distortion.
\newblock {\em Advances in Mathematics}, 400:108261, 2022.

\bibitem{bartl2025empirical}
D.~Bartl and S.~Mendelson.
\newblock Empirical approximation of the gaussian distribution in
  {$\mathbb{R}^d$}.
\newblock {\em Advances in Mathematics}, 460:110041, 2025.

\bibitem{bartl2022structure}
D.~Bartl and S.~Mendelson.
\newblock Structure preservation via the {W}asserstein distance.
\newblock {\em Journal of Functional Analysis}, 288:110810, 2025.

\bibitem{bartl2025uniform}
D.~Bartl and S.~Mendelson.
\newblock A uniform {D}voretzky--{K}iefer--{W}olfowitz inequality.
\newblock {\em Probability Theory and Related Fields}, pages 1--40, 2025.

\bibitem{boedihardjo2025sharp}
M.~T. Boedihardjo.
\newblock Sharp bounds for max-sliced wasserstein distances.
\newblock {\em Foundations of Computational Mathematics}, pages 1--32, 2025.

\bibitem{boucheron2005moment}
S.~Boucheron, O.~Bousquet, G.~Lugosi, and P.~Massart.
\newblock Moment inequalities for functions of independent random variables.
\newblock {\em Annals of Probability}, 33(2):514--560, 2005.

\bibitem{figalli2021invitation}
A.~Figalli and F.~Glaudo.
\newblock {\em An {I}nvitation to {O}ptimal {T}ransport, {W}asserstein
  {D}istances, and {G}radient {F}lows}.
\newblock EMS Textbooks in Mathematics, 2021.

\bibitem{gine2006concentration}
E.~Gin{\'e} and V.~Koltchinskii.
\newblock Concentration inequalities and asymptotic results for ratio type
  empirical processes.
\newblock {\em The Annals of Probability}, 34:1143--1216, 2006.

\bibitem{guedon2017interval}
O.~Gu{\'e}don, A.~E. Litvak, A.~Pajor, and N.~Tomczak-Jaegermann.
\newblock On the interval of fluctuation of the singular values of random
  matrices.
\newblock {\em Journal of the European Mathematical Society}, 19(5), 2017.

\bibitem{jirak2025concentration}
M.~Jirak, S.~Minsker, Y.~Shen, and M.~Wahl.
\newblock Concentration and moment inequalities for sums of independent
  heavy-tailed random matrices.
\newblock {\em Probability Theory and Related Fields}, pages 1--28, 2025.

\bibitem{koltchinskii2017concentration}
V.~Koltchinskii and K.~Lounici.
\newblock Concentration inequalities and moment bounds for sample covariance
  operators.
\newblock {\em Bernoulli}, pages 110--133, 2017.

\bibitem{latala1997estimation}
R.~Lata{\l}a.
\newblock Estimation of moments of sums of independent real random variables.
\newblock {\em Annals of Probability}, 25(3):1502--1513, 1997.

\bibitem{ledoux1991probability}
M.~Ledoux and M.~Talagrand.
\newblock {\em Probability in {B}anach {S}paces: isoperimetry and processes},
  volume~23.
\newblock Springer Science \& Business Media, 1991.

\bibitem{lugosi2024multivariate}
G.~Lugosi and S.~Mendelson.
\newblock Multivariate mean estimation with direction-dependent accuracy.
\newblock {\em Journal of the European Mathematical Society}, 26(6):2211--2247,
  2024.

\bibitem{mendelson2010empirical}
S.~Mendelson.
\newblock Empirical processes with a bounded $\psi_1$ diameter.
\newblock {\em Geometric and Functional Analysis}, 20(4):988--1027, 2010.

\bibitem{mendelson2016upper}
S.~Mendelson.
\newblock Upper bounds on product and multiplier empirical processes.
\newblock {\em Stochastic Processes and their Applications},
  126(12):3652--3680, 2016.

\bibitem{mendelson2007reconstruction}
S.~Mendelson, A.~Pajor, and N.~Tomczak-Jaegermann.
\newblock Reconstruction and subgaussian operators in asymptotic geometric
  analysis.
\newblock {\em Geometric and Functional Analysis}, 17(4):1248--1282, 2007.

\bibitem{mendelson2014singular}
S.~Mendelson and G.~Paouris.
\newblock On the singular values of random matrices.
\newblock {\em Journal of the European Mathematical Society}, 16(4):823--834,
  2014.

\bibitem{olea2022generalization}
J.~L.~M. Olea, C.~Rush, A.~Velez, and J.~Wiesel.
\newblock On the generalization error of norm penalty linear regression models.
\newblock {\em Annals of Statistics, to appear}, 2025.

\bibitem{pisier1999volume}
G.~Pisier.
\newblock {\em The volume of convex bodies and {B}anach space geometry},
  volume~94.
\newblock Cambridge University Press, 1999.

\bibitem{talagrand1987regularity}
M.~Talagrand.
\newblock Regularity of gaussian processes.
\newblock {\em Acta Mathematica}, 159:99--149, 1987.

\bibitem{talagrand2022upper}
M.~Talagrand.
\newblock {\em Upper and lower bounds for stochastic processes: decomposition
  theorems}, volume~60.
\newblock Springer Nature, 2022.

\bibitem{tikhomirov2018sample}
K.~Tikhomirov.
\newblock Sample covariance matrices of heavy-tailed distributions.
\newblock {\em International Mathematics Research Notices},
  2018(20):6254--6289, 2018.

\bibitem{vershynin2018high}
R.~Vershynin.
\newblock {\em High-dimensional probability: An introduction with applications
  in data science}, volume~47.
\newblock Cambridge university press, 2018.

\bibitem{villani2021topics}
C.~Villani.
\newblock {\em Topics in optimal transportation}, volume~58.
\newblock American Mathematical Soc., 2021.

\bibitem{walker1968note}
A.~Walker.
\newblock A note on the asymptotic distribution of sample quantiles.
\newblock {\em Journal of the Royal Statistical Society Series B: Statistical
  Methodology}, 30(3):570--575, 1968.

\end{thebibliography}

\end{document}